\documentclass[a4paper, 9pt]{article}
\usepackage{latexsym}
\usepackage{amsmath}
\usepackage{graphics}
\usepackage{amsthm}
\usepackage{helvet}
\usepackage{amssymb}
\usepackage{bm}
\usepackage[all]{xy}
\newfont{\Fr}{eufm10}
\newfont{\Sc}{eusm10}
\newfont{\Bb}{msbm10}
\newfont{\Am}{msam10}
\newfont{\am}{msam7}
\numberwithin{equation}{section}
\newtheorem{theorem}{Theorem}[section]
\newtheorem{proposition}[theorem]{Proposition}
\newtheorem{lemma}[theorem]{Lemma}
\newtheorem{corollary}[theorem]{Corollary}
\newtheorem{claim}[theorem]{Claim}

\newtheorem{ftheorem}{Theorem}{\bf}{\it}
{\bf}{\it}
\theoremstyle{definition}
\newtheorem{definition}[theorem]{Definition}

\newtheorem{fdefinition}[ftheorem]{Definition}{\bf}{\rm}
\theoremstyle{remark}

\newtheorem{remark}[theorem]{Remark}

\newtheorem{fexample}[ftheorem]{Example}{\it}{\rm}
\newcommand{\Hom}{\mbox{\rm Hom}}

\newcommand{\MID}{\! \! \mid}

\newcommand{\h}{\mathfrac{\h}}

\title{Deformations of nilpotent cones and Springer correspondences}
\author{Syu \textsc{Kato} \footnote{Research Institute for Mathematical Sciences, Kyoto University, Oiwake Kita-Shirakawa Sakyo Kyoto 606-8502, Japan. \tt{E-mail:kato@kurims.kyoto-u.ac.jp}} \footnote{Current Address: Department of Mathematics, Kyoto University, Oiwake Kita-Shirakawa Sakyo Kyoto 606-8502, Japan. \tt{E-mail:syuchan@math.kyoto-u.ac.jp}} \footnote{The author was partially supported by JSPS Grant-in-Aid for Young Scientists (B) 20-740011 during this research.}}

\begin{document}
\maketitle

\begin{center}
{\small Dedicated to the long-standing friendship of Ken-ichi Shinoda and Toshiaki Shoji on the occasion of their 60th birthdays}
\end{center}

\begin{abstract}
Let $G = \mathop{Sp} ( 2n )$ be the symplectic group over $\mathbb Z$. We present a certain kind of deformation of the nilpotent cone of $G$ with $G$-action. This enables us to make direct links between the Springer correspondence of $\mathfrak{sp} _{2n}$ over $\mathbb C$, that over characteristic two, and our exotic Springer correspondence. As a by-product, we obtain a complete description of our exotic Springer correspondence.
\end{abstract}

\section*{Introduction}
Let $G = \mathop{Sp} ( 2n )$ be the symplectic group over $\mathbb Z$. Let $\Bbbk$ be an algebraically closed field. Let $\mathfrak g$ be the Lie algebra of $G$ defined over $\mathbb Z$. Let $\mathcal N$ denote the subscheme of nilpotent elements of $\mathfrak g$. Let $G _{\Bbbk}$, $\mathfrak g _{\Bbbk}$, and $\mathcal N _{\Bbbk}$ denote the specializations of $G$, $\mathfrak g$, and $\mathcal N$ to $\Bbbk$, respectively. 

Springer \cite{Spr} defines a correspondence between the set of $G _{\Bbbk}$-orbits in $\mathcal N _{\Bbbk}$ and a certain set of Weyl group representations (with a basis) when $\mathsf{char} \Bbbk$ is good (i.e. not equal to $2$). This correspondence, together with the so-called ``$A$-group data", lifts to a one-to-one correspondence.

This story is later deepened in two ways. One is Lusztig's generalized Springer correspondence \cite{L1}, which serves as a basis of his theories on Chevalley groups. The other is Joseph's realization \cite{J}, which serves a model of the structure of the primitive spectrum of the enveloping algebra of $\mathfrak g _{\mathbb C}$.

In our previous paper \cite{K1}, we found that a certain Hilbert nilcone $\mathfrak N$ gives a variant of one aspect of the above mentioned Lusztig's theory (c.f. \cite{KL2} and \cite{L2}). Quite unexpectedly, our correspondence gives a one-to-one correspondence without the ``$A$-group data", which is needed in the original Springer correspondence for Weyl groups of type $C$. Therefore, it seems natural to seek some meaning of $\mathfrak N$.

The main theme of this paper is to give one explanation of $\mathfrak N$. Roughly speaking, our conclusion is that $\mathfrak N$ is a model of $\mathcal N _{\mathbb F _2}$ over $\mathbb Z$, which is ``better" than $\mathcal N$ in a certain sense.

To see what we mean by this, we need a more precise formulation: Let $T$ be a maximal torus of $G$. We define the Weyl group of $(G, T)$ as $W := N _G ( T ) / T$. We denote the set of isomorphism classes of irreducible representations of $W$ by $W ^{\vee}$. Let $V _1$ be the vector representation. Put $V _2 := \wedge ^2 V _1$. We denote $V _1 \oplus V _2$ by $\mathbb V$. Let $\epsilon _1, \ldots, \epsilon _n$ be the standard choice of $T$-weight basis of $G$ (see eg. Bourbaki \cite{B}). We denote the ``positive part" of $\mathfrak g$ and $\mathbb V$ by $\mathfrak n$ and $\mathbb V ^+$, respectively (c.f. \S \ref{NT}). Let $\mathfrak N$ be the Hilbert nilcone of $( G, \mathbb V )$ over $\mathbb Z$. We have a natural map
$$\nu : G \times ^B \mathbb V ^+ \longrightarrow \mathfrak N,$$
which we regard as a counter-part of the Springer resolution.

\begin{ftheorem}\label{fnormal}
The variety $\mathfrak N$ is normal and flat over $\mathbb Z$. Moreover, the number of $G _{\Bbbk}$-orbits of $\mathfrak N _{\Bbbk}$ is independent of the characteristic of $\Bbbk$.
\end{ftheorem}

\begin{ftheorem}\label{ffam}
Let $\Bbbk = \overline{\mathbb F} _2$. There exists a $G _{\Bbbk}$-equivariant flat family $\pi : \mathcal N _S \longrightarrow \mathbb A ^1 _{\Bbbk}$ with the following properties:
\begin{enumerate}
\item We have $\pi ^{-1} ( t ) \cong \mathcal N _{\Bbbk}$ for $t \neq 0$;
\item There exists an isogeny $\mathsf{F} _1 : \mathfrak N _{\Bbbk} \longrightarrow \pi ^{-1} ( 0 )$, which is an endomorphism as varieties.
\end{enumerate}
Moreover, for a $G _{\Bbbk}$-orbit $\mathcal O _{\Bbbk} \subset \mathfrak N _{\Bbbk}$, there exists a flat subfamily of single $G _{\Bbbk}$-orbits $\mathcal O _S \subset \mathcal N _S$ such that $\mathcal O _S \cap \pi ^{-1} ( 0 ) = \mathsf{F} _1 ( \mathcal O _{\Bbbk} )$.
\end{ftheorem}

Theorem \ref{fnormal} claims that our variety $\mathfrak N$ behaves well with respect to the specializations. Theorem \ref{ffam} claims that we can regard $\mathfrak N$ as a model of $\mathcal N _{\Bbbk}$ in a certain sense.

To illustrate these, let us describe the orbit correspondence of Theorem \ref{ffam}, together with the corresponding Springer correspondences:

\begin{fexample}[The orbit correspondence for $n = 2$]\label{eo2}
Let $\mathfrak t _{\mathbb C}$ be a Cartan subalgebra of $\mathfrak g _{\mathbb C}$. Let $R = \{ \pm \epsilon _i \pm \epsilon _j \} _{1\le i,j\le 2} \backslash \{ 0 \} \subset \mathfrak t ^* _{\mathbb C}$ be the set of roots of $\mathfrak g _{\mathbb C}$. We choose its positive simple roots as $\alpha _1 = \epsilon _1 - \epsilon _2$ and $\alpha _2 = 2 \epsilon _2$. Let $\mathbf x [ \lambda ] \in \mathfrak g$ and $\mathbf v [ \lambda ] \in \mathbb V$ be $T$-eigenvectors with $T$-weight $\lambda$, respectively. We refer the Springer correspondence of $\mathcal N$ by ordinary and that of $\mathfrak N$ by exotic. Then, representatives of $G _{\Bbbk}$-orbits of $\mathcal N _{\Bbbk}$ and $\mathfrak N _{\Bbbk}$ corresponding to each member of $W ^{\vee}$ are:
\begin{center}
\begin{tabular}{cc|c|c|c}
$W ^{\vee}$ & dim. & ordinary ($\mathsf{char} \Bbbk \neq 2$) & ordinary ($\mathsf{char} \Bbbk = 2$) & exotic \\ \hline
sign & 1 & 0 & 0 & 0 \\
Ssign & 1 & $\mathbf x [ 2 \epsilon _1 ]$ & $\mathbf x [ 2 \epsilon _1 ]$ & $\mathbf v [ \epsilon _1 ]$ \\
Lsign & 1 & $\mathbf x [ \alpha _1 ]$ & $\mathbf x [ \alpha _1 ]$ & $\mathbf v [ \alpha _1 ]$ \\
regular & 2 & $\mathbf x [ \alpha _1 ]$& $\mathbf x [ \alpha _1 ] + \mathbf x [ 2 \epsilon _1 ]$ & $\mathbf v [ \alpha _1 ] + \mathbf v [ \epsilon _1 ]$ \\
triv & 1 & $\mathbf x [ \alpha _1 ] + \mathbf x [ \alpha _2 ]$ & $\mathbf x [ \alpha _1 ] + \mathbf x [ \alpha _2 ]$ & $\mathbf v [ \alpha _1 ] + \mathbf v [ \epsilon _2 ]$ \\\hline
\end{tabular}
\end{center}
\end{fexample}

Theorem \ref{ffam} gives an isogeny between the Springer fibers of $\mathcal N _{\Bbbk}$ and $\mathfrak N _{\Bbbk}$ when $\mathsf{char} \Bbbk = 2$. This implies that the Springer correspondences associated to $\mathcal N _{\Bbbk}$ and $\mathfrak N _{\Bbbk}$ must coincide up to scalar multiplication of their basis.

To see this phenomenon more closely, we employ the Joseph model of the Springer representations. Following Joseph \cite{J}, we define the orbital variety attached to a $G _{\mathbb C}$-orbit $\mathbb O _{\mathbb C} \subset \mathcal N _{\mathbb C}$ as an irreducible component of the intersection $\mathbb O _{\mathbb C} \cap \mathfrak n _{\mathbb C}$. Let us denote the set of orbital varieties attached to $\mathbb O _{\mathbb C}$ by $\mathrm{Comp} ( \mathbb O _{\mathbb C} )$. Similarly, let $\mathcal O _{\mathbb C} \subset \mathfrak N _{\mathbb C}$ be a $G _{\mathbb C}$-orbit and let $\mathrm{Comp} ( \mathcal O _{\mathbb C} )$ be the set of irreducible components of $\mathcal O _{\mathbb C} \cap \mathbb V ^+ _{\mathbb C}$. We also call a member of $\mathrm{Comp} ( \mathcal O _{\mathbb C} )$ an orbital variety (attached to $\mathcal O _{\mathbb C}$).

Joseph found that the leading terms of $T$-equivariant Hilbert polynomials of $\mathrm{Comp} ( \mathbb O _{\mathbb C} )$ yield an irreducible $W$-module which is contained in the Springer representation attached to $\mathbb O _{\mathbb C}$. These polynomials are usually called the Joseph polynomials. In representation-theoretic context (originally pursued by Joseph), $\mathrm{Comp} ( \mathcal O _{\mathbb C} )$ is precisely the set of Lagrangian subvarieties of $\mathcal O _{\mathbb C}$ corresponding to highest weight modules of $\mathfrak g _{\mathbb C}$ and these polynomials build up ``asymptotic Goldie ranks" of primitive ideals of $U ( \mathfrak g _{\mathbb C} )$ (see eg. \cite{Jo}).

In view of Hotta \cite{Hot} (see also Joseph \cite{J2} or Chriss-Ginzburg \cite{CG}), it is straightforward to see that Joseph's construction extends to the case of our exotic Springer correspondence. In particular, we have the notion of Joseph polynomials attached to each orbit of $\mathfrak N$.

\begin{ftheorem}\label{fJoscomp}
Let $\Bbbk = \overline{\mathbb F} _2$. Let $\mathbb O _{\mathbb C}$ be a $G _{\mathbb C}$-orbit of $\mathcal N _{\mathbb C}$. Then, there exists $G$-stable locally closed subsets $\mathbb O \subset \mathcal N$ and $\mathcal O \subset \mathfrak N$ such that
\begin{enumerate}
\item We have $\mathbb O _{\mathbb C} = \mathbb O \otimes \mathbb C$;
\item The variety $\mathcal O \otimes \Bbbk$ is a single $G _{\Bbbk}$-orbit which corresponds to a unique dense open $G _{\Bbbk}$-orbit of $\mathbb O \otimes \Bbbk$;
\item The Joseph polynomials of $\mathcal O _{\mathbb C}$ and that of $\mathbb O _{\mathbb C}$ are equal up to scalar.
\end{enumerate}
\end{ftheorem}
It may worth to mention that there exists some orbit $\mathcal O$ of $\mathfrak N$ which does not correspond to an orbit of $\mathcal N _{\mathbb C}$. In this case, our version of Joseph polynomials realize a Weyl group representation which cannot be realized by the usual Joseph polynomials. To illustrate this phenomenon, we compare Joseph polynomials for $\mathop{Sp} ( 4 )$:

\begin{fexample}[Joseph polynomials for $n = 2$]
Keep the setting of Example \ref{eo2}. By the natural $W$-action on $\mathfrak t ^* _{\mathbb C}$, we have a $W$-action on $\mathbb C [ \mathfrak t _{\mathbb C} ] = \mathbb C [ \epsilon _1, \epsilon _2 ]$ preserving each degree. Each Joseph polynomial belongs to $\mathbb C [ \mathfrak t _{\mathbb C} ]$ and hence admits a $W$-action. The list of Joseph polynomials corresponding to each member of $W ^{\vee}$ via Example \ref{eo2} is:
\begin{center}
\begin{tabular}{cc|c|c|c}
$W ^{\vee}$ & dim. & ordinary ($\mathsf{char} \Bbbk \neq 2$) & ordinary ($\mathsf{char} \Bbbk = 2$) & exotic \\ \hline
sign & 1 & $4 \epsilon _1 \epsilon _2 ( \epsilon _1 ^2 - \epsilon _2 ^2 )$ & $4 \epsilon _1 \epsilon _2 ( \epsilon _1 ^2 - \epsilon _2 ^2 )$ & $\epsilon _1 \epsilon _2 ( \epsilon _1 ^2 - \epsilon _2 ^2 )$ \\
Ssign & 1 & $2 ( \epsilon _1 ^2 - \epsilon _2 ^2 )$& $2 ( \epsilon _1 ^2 - \epsilon _2 ^2 )$ & $( \epsilon _1 ^2 - \epsilon _2 ^2 )$ \\
Lsign & 1 & N/A & $4 \epsilon _1 \epsilon _2$ & $\epsilon _1 \epsilon _2$ \\
regular & 2 & $\alpha _1, 2 \epsilon _2$ & $\alpha _1, 2 \epsilon _2$ & $\alpha _1, \epsilon _2$ \\
triv & 1 & 1 & 1 & 1 \\\hline
\end{tabular}
\end{center}
Note that the $\mathbb C$-span of polynomials in each entry of the above table recovers the original member of $W ^{\vee}$.
\end{fexample}

Since our exotic Springer correspondence shares a similar flavor with the usual Springer correspondence of type $A$, it is natural to expect a combinatorial description. To state this, we need:

\begin{fdefinition}
Let $(\mu, \nu)$ be a pair of partitions such that $\left| \mu \right| + \left| \nu \right| = n$. For a partition $\lambda$, we put $\lambda _i ^{<} := \sum _{j < i} \lambda _j$ and $\lambda _i ^{\le} := \sum _{j \le i} \lambda _j$. We define
\begin{align*}
\mathsf{D} ^0 _i ( \mu ) := \prod _{\mu _i ^< < k < l \le \mu _i ^{\le}} ( \epsilon _k ^2 - \epsilon _l ^2 ), & \ \mathsf{D} ^+ _i ( \mu, \nu ) := \prod _{\nu _i ^< < k < l \le \nu _i ^{\le}} ( \epsilon _{k + \left| \mu \right|} ^2 - \epsilon _{l + \left| \mu \right|} ^2 ), \text{ and }\\
\mathsf{D} ( \mu, \nu ) := & \prod _{i > \left| \mu \right|} \epsilon _i \times \prod _{i = 1} ^{\infty} \mathsf{D} ^0 _i ( \mu ) \mathsf{D} ^+ _i ( \mu, \nu ).
\end{align*}
\end{fdefinition}

\begin{ftheorem}\label{fdetsp}
For each $G$-orbit $\mathcal O$ of $\mathfrak N$, there exists a pair of partitions $(\mu, \nu)$ and $\mathfrak X \in \mathrm{Comp} ( \mathcal O )$ such that the Joseph polynomial of $\mathfrak X$ is a scalar multiplication of $\mathsf{D} ( \mu, \nu )$.
\end{ftheorem}

Since $(\mu, \nu)$ in Theorem \ref{fdetsp} is easily computable, this completes a determination of our exotic Springer correspondence. Taking account into Theorem \ref{fJoscomp}, we have determined some special Joseph polynomials which we cannot compute easily from their naive definitions.

The organization of this paper is as follows:

In \S 1, we fix convention and introduce our variety $\mathfrak N$. Then, we describe its set of defining equations in \S 2. Our system of defining equations is explicit and behaves nice with respect to the restriction to certain linear subvariety. These facts enable us to prove that $\mathfrak N$ is normal in \S 3. This proves the first part of Theorem \ref{fnormal}. Also, we introduce a parameterization of orbits of $\mathfrak N$ over $\mathbb Z$ or $\Bbbk$. The \S4 contains the main observations of this paper. Namely, we observe:
\begin{itemize}
\item the adjoint representation $\mathfrak g$ of a symplectic group over characteristic two is not irreducible;
\item this reducibility enables one to define a natural deformation of $\mathfrak g$, and its subvariety $\mathcal N$ in characteristic two;
\item the special fiber such that (the deformation of) $\mathfrak g$ becomes decomposable is isogenous to $\mathbb V$;
\item the above three observations are sufficient to construct a ``deformation" from $\mathcal N$ (general fiber) to $\mathfrak N$ (special fiber) in characteristic two.
\end{itemize}
These observations enable us to prove Theorem \ref{ffam}. In \S 5, we define a particular element of the Weyl group attached to each orbit of $\mathfrak N$ which cut-off a particular part of the orbit. In \S 6, we see that every orbit of $\mathfrak N _{\Bbbk}$ extends to an orbit of $\mathfrak N$ in order to prove the second part of Theorem \ref{fnormal}. The \S\S 7--8 are devoted to the equi-dimensionality of the orbital varieties attached to $\mathfrak N$. Its proof is nothing but a minor modification of the Steinberg-Spaltenstein-Joseph theorem, which we present here for the reference purpose. (So I claim no originality here.) These are preparatory steps to the later sections. In \S 9, we use the results in the previous sections to prove Theorem \ref{fJoscomp}. With the help of previous sections and Joseph's theory, the only missing piece boils down to the rigidity of the torus character. In \S 10, we construct a special orbital variety from an orbit of $\mathfrak N$ in order to prove Theorem \ref{fdetsp}. The main difficulty in the couse of its proof is that we cannot expect some orbital variety to be a linear subspace contrary to the type $A$ case. We make a trick coming from the symmetry of Joseph polynomials to avoid this difficulty.\\

With the technique developed in this paper, a similar construction applied to $G ^{\vee} = \mathop{SO} ( 2n + 1 )$ yields an analogue of Theorem \ref{fJoscomp} for special representations of the Weyl groups of $\mathop{Sp} ( 2n )$ and $\mathop{SO} ( 2n + 1 )$. However, as the work of Tian Xue \cite{X} suggests, there might be a better formulation of the connection between these two cases.\\

Finally, one word of caution is in order. We work not over $\mathrm{Spec} \mathbb Z$ but a neighborhood of $\mathrm{Spec} \overline{\mathbb F} _2$ in the main body of this paper. The reason is that two is the only bad prime for symplectic groups and the corresponding statements are more or less trivial (or inexistent) with respect to the reduction to the other primes.

\section{Preliminaries}
\subsection{Convention}
Consider a ring
$$A : = \mathbb Z [ p ^{-1}, \zeta _N ; p, N \in \mathbb Z _{> 0}, ( p, 2 ) = 1, \zeta _N ^{2 ^N - 1} = 1 ] \subset \overline{\mathbb Q}.$$
This is a local ring with a unique maximal ideal $( 2 )$. Let $\mathbb K$ be the quotient field of $A$ and let $\Bbbk$ be the residual field of $A$. We have $\Bbbk = \overline{\mathbb F} _2$.

For a partition $\lambda = ( \lambda _1, \lambda _2, \ldots )$, we define $\lambda _i ^{<} := \sum _{j < i} \lambda _j$ and $\lambda _i ^{>} := \sum _{j > i} \lambda _j$ for each $i$. We also use the notation $\lambda _i ^{\le} := \lambda _{i + 1} ^{<}$ and $\lambda _i ^{\ge} := \lambda _{i - 1} ^>$. We put $\left| \lambda \right| := ( \lambda ) _1 ^{\ge}$. We denote the dual partition of $\lambda$ by ${} ^t \lambda$.

For a scheme $\mathcal X$ over $A$, we denote its specializations to $\Bbbk$ and $\mathbb K$ by $\mathcal X _{\Bbbk}$ and $\mathcal X _{\mathbb K}$, respectively. In addition, assume that $\mathcal X$ admits an action of a group scheme $G$ over $A$. By a $G$-orbit on $\mathcal X$, we refer a flat subfamily $\mathcal O$ of $\mathcal X$ over $A$ such that $\mathcal O _{\mathbb K}$ is a single $G _{\mathbb K}$-orbit. For a map of commutative rings $A \rightarrow D$, we define $\mathcal X ( D )$ the set of $D$-valued points of $\mathcal X$. For each $e \in G ( A )$, we denote by $\mathcal X ^e$ the $e$-fixed part of $\mathcal X$. We denote by $H _{\bullet} ( \mathcal X, \mathbb C )$ the Borel-Moore homology of $\mathcal X _{\mathbb C}$. 

We understand that the intersection $\cap$ of two (sub-)schemes are set-theoretic. (I.e. we consider the reduced part of the scheme-theoretic intersection.) The scheme-theoretic intersection is denoted by $\dot{\cap}$.

For a scheme $\mathcal Y$ over $\Bbbk$, we denote its (geometric) Frobenius endomorphism by $\mathsf{Fr}$. Here geometric means that the induced map $\mathsf{Fr} ^* : \mathcal O _{\mathcal Y} \rightarrow \mathcal O _{\mathcal Y}$ is $\Bbbk$-linear and (suitable) local coordinates are changed to its $2$nd power.

\subsection{Notation and Terminology}\label{NT}

Let $G = \mathop{Sp} ( 2n, A )$ be the symplectic group of rank $n$ over $A$. Let $B \supset T$ be its Borel subgroup and a maximal torus defined over $A$. Put $N := [ B, B ]$. Denote the opposite unipotent radical of $N$ (with respect to $T$) by $N ^-$. Let $W := N _G ( T ) / T$ be the Weyl group of $G$. We denote by $X ^* ( T )$ the weight lattice of $T$. Let $R$ be the set of roots of $( G, T )$ with its positive part $R ^+$ determined by $B$. Consider an $A$-module $V _1 := A ^{2n}$, for which $G$ acts by the multiplication of matrices. Let $V _2 := \wedge ^2 V _1$ $(\subset \wedge ^2 ( V _1 ) _{\mathbb C} )$ be the $A$-module with the natural $G$-module structure. Let $\mathfrak g$ be the Lie algebra of $G$ over $A$, whose integral structure is $\mathsf{Sym} ^2 V _1 = ( V _1 \otimes _A V _1 ) ^{\mathfrak S _2}$. Let $\mathfrak b, \mathfrak t, \mathfrak n$ be the intersections of Lie algebras corresponding to $B _{\mathbb C}, T _{\mathbb C}, N _{\mathbb C}$ with $\mathfrak g$ inside of $\mathfrak g _{\mathbb C}$, respectively.

Fix a $\mathbb Z$-basis $\epsilon _1, \ldots, \epsilon _n$ of $X ^* ( T )$ such that $R ^+= \{ \epsilon _i \pm \epsilon _j \} _{i < j} \cup \{ 2 \epsilon _i \} _i \subset X ^* ( T )$. For each $i$, we put $\alpha _i := \epsilon _i - \epsilon _{i + 1}$ ($1 \le i < n$), $2 \epsilon _n$ ($i = n$). Let $s _i$ be the reflection of $W$ corresponding to $\alpha _i$. Let $\ell : W \rightarrow \mathbb Z$ denote the length function on $W$ with respect to $s _1, \ldots, s _n$.

We put $\mathbb V := V _1 \oplus V _2$. For a $T$-weight $\lambda \neq 0$, let $\mathbb V [ \lambda ]$ be the $T$-eigenpart  of $\mathbb V$ with its weight $\lambda$. Let $\mathbb V ^+$ be the sum of $T$-weight spaces of $\mathbb V$ with its weights in $\mathbb Q _{\ge 0} R ^+ - \{ 0 \}$. For a $T$-subrepresentation $V \subset \mathbb V$, we denote by $\Psi ( V )$ the set of $T$-weights $\lambda$ with $\mathbb V [ \lambda ] \neq \{ 0 \}$.

For each $w \in W$, we denote (one of) its lift by $\dot{w} \in N _G ( T )$. For a $T$-stable subset $\mathcal S$ in $\mathbb V$ or $\mathfrak g$, we define ${} ^w \mathcal S := \dot{w} \mathcal S$.

We denote the flag variety $G / B$ by $\mathcal B$.

Let $\mathfrak N$ be the $G$-subscheme of $\mathbb V$ defined by the positive degree part $A [ \mathbb V ] ^G _+$ of $A [ \mathbb V ] ^G$.

Let $\mathcal N$ be the space of ad-nilpotent elements of $\mathfrak g$.

\begin{theorem}[Hesselink \cite{Hes}]\label{HLcls}
We have $\mathcal N ( \Bbbk ) = \bigcup _{g \in G ( \Bbbk )} \mathrm{Ad} ( g ) \mathfrak n ( \Bbbk )$.
\end{theorem}

Theorem \ref{HLcls} implies that the natural map (the Springer resolution over $A$)
$$\mu : G \times ^{B} \mathfrak n \longrightarrow \mathcal N$$
is surjective at the level of points.

For a $G _{\Bbbk}$-module $V$ over $\Bbbk$, we define its Frobenius twist $V ^{[1]}$ as the composition
$$G _{\Bbbk} \longrightarrow \mathop{GL} ( V ) \stackrel{\mathsf{Fr}}{\longrightarrow} \mathop{GL} ( V ) \subset \mathrm{End} _{\Bbbk} ( V ).$$

\section{Defining equations of $\mathfrak N$}
Let $e \in T$ be an element such that $\epsilon _i ( e ) = c$ (for every $1 \le i \le n$), where $c \in A$ is an element with sufficiently high order after taking modulo two. In particular, we assume $Z _G ( e ) \cong \mathop{GL} ( n, A )$, $V _1 ^e = \{ 0 \}$, and $V _2 ^e \cong \mathrm{Mat} ( n, A )$. Put $G _0 := Z _G ( e )$. Consider a direct sum decomposition
\begin{eqnarray}
\mathfrak g = \mathfrak g _{- 2} \oplus \mathfrak g _0 \oplus \mathfrak g _{2}, \text{ and } \mathbb V = \mathbb V _{-2} \oplus \mathbb V _{-1} \oplus \mathbb V _0 \oplus \mathbb V _1 \oplus \mathbb V _2\label{edec}
\end{eqnarray}
determined by the eigenvalues of the action of $e$ (indicated as subscript). Here $\mathfrak g _{\pm 2}$ and $\mathbb V _i$ ($- 2 \le i \le 2$) are $G _0$-modules. We have $\mathfrak g _0 \cong \mathbb V _0$ as $G _0$-modules. Let $\mathfrak n _0 := \mathrm{Lie} G _0 \cap \mathfrak n$, which we may regard as a subspace of $\mathbb V _0$. We define $G _{-2}, G _{2}, N _0$ to be the unipotent subgroups of $G$ corresponding to $\mathfrak g _{-2}, \mathfrak g _{2},$ and $\mathfrak n _0$, respectively. We fix an identification $\mathfrak S _n = N _{G _0} ( T ) / T \cong \left< s _i ; i < n\right> \subset W$. We define
$$\mathbf J := \left( \begin{matrix} 0 & - 1 _n \\ 1 _n & 0 \end{matrix} \right).$$

We have $V _2 \cong \mathrm{Alt} ( 2n, A )$ as $GL ( 2n, A )$-modules. Hence, it restricts to a $G$-module isomorphism. Let $\mathsf{Pf}$ be the Pfaffian associated to $X = \{ x _{ij} \} _{ij} = \{ - x _{ji} \} _{ij} \in \mathrm{Alt} ( 2n )$. It is defined as
\begin{eqnarray}
\mathsf{Pf} ( X ) = \frac{1}{n!} \sum _{\sigma} \mathrm{sgn} ( \sigma ) x _{\sigma ( 1 ) \sigma ( 2 )} \cdots x _{\sigma ( 2 n - 1 ) \sigma ( 2 n )},\label{dP}
\end{eqnarray}
where $\sigma$ runs over all permutations of $\mathfrak S _{2n}$ such that $\sigma ( 2 m - 1 ) < \sigma ( 2 m )$ for every $1 \le m \le n$. By using $\mathsf{Pf}$, we define polynomials $1 = P _0, P _1, \ldots, P _n$ on $V _2$ as
$$\sum _{i = 0} ^n t ^{n - i} P _i ( X ) = \mathsf{Pf} ( t \mathbf J - X ) \hskip 14mm ( X \in V _2 \cong \mathrm{Alt} ( 2n )).$$
We have $P _i \in A [ V _2 ] ^G$ by
$$\mathsf{Pf} ( t \mathbf J - X ) = \det ( g ) \mathsf{Pf} ( t \mathbf J - X ) = \mathsf{Pf} ( t g \mathbf J {} ^t g - g X {} ^t g ) = \mathsf{Pf} ( t \mathbf J - g X {} ^t g )$$
for each $g \in G ( \mathbb K )$.

\begin{proposition}\label{defeq}
By means of the $G$-module isomorphism $V _2 \cong \mathrm{Alt} ( 2n )$, the variety $( \mathfrak N \cap V _2 )$ is identified with the common zeros of $P _1, \ldots, P _n$.
\end{proposition}
\begin{proof}
Under the isomorphism $V _2 \cong \mathrm{Alt} ( 2n )$, the subspace $\mathbb V _0 \subset V _2$ corresponds to
$$\mathrm{Alt} ( 2n ) _0 = \{ \{ x _{ij} \} \in \mathrm{Alt} ( 2n ) ; x _{ij} = 0 \text{ for } 1 \le i, j \le n \text{ or } n < i, j \le 2n \}.$$
Substituting them into the definition of Pfaffians, we deduce
$$\mathsf{Pf} ( t \mathbf J - X ) = (- 1) ^n \det ( t 1 _n - Y ) \hskip 5mm \text{where } X = \left( \begin{matrix} 0 & - Y \\ {} ^t Y & 0 \end{matrix} \right).$$
This implies
$$A [ V _2 ] ^G \supset A [ P _i ; i \ge 1 ] \cong A [ \mathbb V [ 0 ] ] ^{\mathfrak S _n}$$
via the restriction map. By the Dadok-Kac classification \cite{DK} table 2, we know that
$$\mathbb C [ V _2 ] ^G \cong \mathbb C [ \mathbb V _0 ] ^{G _0} \cong \mathbb C [ \mathbb V [ 0 ] ] ^{\mathfrak S _n}.$$
This implies that $\{ P _i \} _{i = 1} ^n$ generates the ideal $\left< A [ V _2 ] ^G _+ \right>$ as desired.
\end{proof}

\begin{corollary}
We have $\mathfrak N \cong A [ \mathbb V ] / ( P _i ; 1 \le i \le n)$.
\end{corollary}

\begin{proof}
It is clear from the isomorphism $\mathbb C [ \mathbb V ] ^G \cong \mathbb C [ V _2 ] ^G$, which can be read off from the Dadok-Kac classification (\cite{DK} table 2).
\end{proof}

For the later use, we prove a lemma.

\begin{lemma}\label{pi}
Let $Y \in \mathrm{Alt} ( n )$ and let $Z \in \mathrm{Mat} ( n )$. Then, we have
$$P _i \left( \begin{matrix} Y & Z \\ - {} ^t Z & 0 \end{matrix} \right) = P _i \left( \begin{matrix} 0 & Z \\ - {} ^t Z & 0 \end{matrix} \right)$$
for each $i$.
\end{lemma}

\begin{proof}
By the pigeon hole principle, if a Pfaffian term
$$\mathrm{sgn} ( \sigma ) x _{\sigma ( 1 ) \sigma ( 2 )} \cdots x _{\sigma ( 2 n - 1 ) \sigma ( 2 n )}$$
in $(\ref{dP})$ satisfies $\sigma ( 2 m - 1 ), \sigma ( 2m ) \le n$ for some $m$, then there exists $m ^{\prime}$ such that $\sigma ( 2 m ^{\prime} - 1 ), \sigma ( 2 m ^{\prime} ) > n$. Then, this term cannot contribute to $\mathsf{Pf} \left( \begin{matrix} Y & Z \\ - {} ^t Z & 0 \end{matrix} \right)$, which implies the result.
\end{proof}

%\begin{lemma}\label{sat}
%We have $\mathfrak N ( A ) = \left( G \mathbb V ^+ \right) ( A ) \subset \mathbb V ( A )$.
%\end{lemma}

%\begin{proof}
%Since we have $\mathcal X ( A ) \subset \mathcal X ( \mathbb C )$ for a $A$-scheme $\mathcal X$, it suffices to prove the assertion over $\mathbb C$, which is a consequence of \cite{K1} Theorem 1.2.
%\end{proof}

\section{Geometric construction of $\mathfrak N$}
We retain the setting of the previous section. 

Let $F := G \times ^B \mathbb V ^+$. Consider the map
$$\nu : F = G \times ^B \mathbb V ^+ \longrightarrow \mathbb V.$$
We denote the specialization of $\nu$ to $\mathbb K$ and $\Bbbk$ by $\nu _{\mathbb K}$ and $\nu _{\Bbbk}$, respectively. 
Since the fiber of $\nu$ is naturally isomorphic to a closed subscheme of a flag variety, $\nu$ is projective.

\begin{lemma}\label{satV}
We have $G ( \mathbb V _0 \oplus \mathbb V _1 \oplus \mathbb V _2 ) = G ( \mathbb V [ 0 ] \oplus \mathbb V ^+ ) = \mathbb V$.
\end{lemma}

\begin{proof}
The subgroup $G_0 B \subset G$ is parabolic. It preserves $( \mathbb V _0 \oplus \mathbb V _1 \oplus \mathbb V _2 ) \subset \mathbb V$. Hence, $G ( \mathbb V _0 \oplus \mathbb V _1 \oplus \mathbb V _2 )$ is a closed subscheme of $\mathbb V$. Thus, it suffices to show the corresponding statement over $\mathbb C$. It is well-known that
$$( G _0 ) _{\mathbb C} ( \mathbb V _0 \cap ( \mathbb V [0] \oplus \mathbb V ^+ ) ) _{\mathbb C} = ( \mathbb V _0 ) _{\mathbb C}.$$
Therefore, it suffices to show the map
$$\widetilde{\nu} : G \times ^{B} ( \mathbb V [0] \oplus \mathbb V ^+ ) \longrightarrow \mathbb V$$
extending $\nu$ is surjective after specializing to $\mathbb C$. We have $\mathfrak N _{\mathbb C} = \mathrm{Im} \nu _{\mathbb C}$ by \cite{K1} 1.2 1). By the proof of Proposition \ref{defeq}, we have $\mathbb V [ 0 ] _{\mathbb C} \cap \mathfrak N _{\mathbb C} = \{ 0 \}$. Therefore, we deduce that
$$\dim \mathrm{Im} \widetilde{\nu} _{\mathbb C} \ge \dim \mathfrak N _{\mathbb C} + \dim \mathbb V [ 0 ] _{\mathbb C} = \dim \mathbb V _{\mathbb C}$$
by $\mathbb V [ 0 ] _{\mathbb C} \subset \mathrm{Im} \widetilde{\nu} _{\mathbb C}$. Since $\mathrm{Im} \widetilde{\nu} _{\mathbb C} \subset \mathbb V _{\mathbb C}$ is a closed subscheme, we conclude the result.
\end{proof}

\begin{corollary}\label{satk}
We have $\mathfrak N ( \Bbbk ) = G ( \Bbbk ) \mathbb V ^+ ( \Bbbk ) \subset \mathbb V ( \Bbbk )$. \hfill $\Box$
\end{corollary}

\begin{proof}
By Lemma \ref{satV} and the fact that the map $\widetilde{\nu}$ (borrowed from the proof of Lemma \ref{satV}) is flat over $A$, we deduce that $G ( \Bbbk ) ( \mathbb V _0 \oplus \mathbb V _1 \oplus \mathbb V _2 ) ( \Bbbk ) = \mathbb V ( \Bbbk )$. Since $\mathfrak N ( \Bbbk )$ is $G ( \Bbbk )$-stable, we have
$$\mathfrak N ( \Bbbk ) = G ( \Bbbk ) ( \mathfrak N \cap  \mathbb V _0 \oplus \mathbb V _1 \oplus \mathbb V _2 ) ( \Bbbk ).$$
By Lemma \ref{pi}, we have
$$( \mathfrak N \cap  \mathbb V _0 \oplus \mathbb V _1 \oplus \mathbb V _2 ) ( \Bbbk ) = ( \mathfrak N \cap  \mathbb V _0 )  ( \Bbbk ) \times (  \mathbb V _1 \oplus \mathbb V _2 ) ( \Bbbk ).$$
By the description of the nilpotent cone of $\mathfrak g _0 ( \Bbbk ) \cong \mathbb V _0 ( \Bbbk )$, we deduce
$$( \mathfrak N \cap  \mathbb V _0 )  ( \Bbbk ) = G _0 ( \Bbbk ) (  \mathbb V _0 \cap \mathbb V ^+ )  ( \Bbbk ).$$
Taking account into the fact that $\mathbb V _1$ and $\mathbb V _2$ are $G _0$-stable, we conclude the result.
\end{proof}

\begin{lemma}\label{exotic semi-small}
The map $\nu$ is semi-small with respect to the stratification given by $G$-orbits. The same holds for $\nu _{\Bbbk}$ provided if $G ( \Bbbk ) \backslash \mathfrak N ( \Bbbk ) < \infty$.
\end{lemma}

\begin{proof}
This is a straight-forward generalization of the results in \cite{K1} \S 1.1.
\end{proof}

\begin{remark}\label{exotic str semi-small}
By a result of Borho-MacPherson (c.f. \cite{CG} \S 8.9), our exotic Springer correspondence (a bijection between $G _{\mathbb C}$-orbits of $\mathfrak N _{\mathbb C}$ and irreducible representations of $W$) implies that $\nu _{\mathbb C}$ must be strictly semi-small. (Otherwise there must be some $G _{\mathbb C}$-orbit which does not correspond to an irreducible representation of $W$.)
\end{remark}

\begin{proposition}\label{LI}
The differentials $d P _1, \ldots, d P _n$ of the polynomials $P _1, \ldots, P _n$ are linearly independent up to codimension two subscheme of $\mathfrak N$.
\end{proposition}

\begin{proof}
By \cite{K1} 1.2 6), a non-dense orbit of $G _{\mathbb K}$ in $\mathfrak N _{\mathbb K}$ has codimension at least two. Hence it suffices to check the assertion for an open subset of $\mathfrak N _{\Bbbk}$. By Lemma \ref{satV}, we can replace $\mathfrak N _{\Bbbk}$ with $\mathfrak N _{\Bbbk} \cap ( \mathbb V _0 \oplus \mathbb V _1 \oplus \mathbb V _2 ) _{\Bbbk}$. By Lemma \ref{pi} and the proof of Proposition \ref{defeq}, it suffices to prove the assertion on the dense open $( G _0 ) _{\Bbbk}$-orbit of $( \mathfrak N _{\Bbbk} \cap (\mathbb V _0 ) _{\Bbbk} )$, which is the regular nilpotent orbit of $\mathop{GL} ( n ) _{\Bbbk}$. This is well-known (or is easily checked).
\end{proof}

\begin{corollary}
The scheme $\mathfrak N$ is regular in codimension one.
\end{corollary}

\begin{proof}
The scheme $\mathfrak N$ is a complete intersection up to codimension two locus.
\end{proof}

\begin{proposition}\label{CM}
The scheme $\mathfrak N$ is Cohen-Macaulay.
\end{proposition}

\begin{proof}
We have
$$A [ \mathbb V [ 0 ] ] ^{\mathfrak S _n} \cong A [ \mathbb V ] ^G = A [ P _1, \ldots P _n ].$$
As a consequence, we deduce that $A [ \mathbb V [ 0 ] ]$ is a free $A [ \mathbb V [ 0 ] ] ^{\mathfrak S _n}$-module by the Pittie-Steinberg theorem. It follows that the multiplication map $A [ \mathbb V / \mathbb V [ 0 ] ] \otimes _A A [ \mathbb V ] ^G \rightarrow A [ \mathbb V ]$ equip $A [ \mathbb V ]$ the structure of a free $A [ \mathbb V / \mathbb V [ 0 ] ] \otimes _A A [ \mathbb V ] ^G$-module. We have $A [ \mathfrak N ] \cong A [ \mathbb V ] \otimes _{A [ \mathbb V ] ^G} A$. This is a free $A [ \mathbb V / \mathbb V [ 0 ] ]$-module, which implies that $\mathfrak N$ is Cohen-Macaulay.
\end{proof}

\begin{corollary}
The scheme $\mathfrak N$ is flat over $A$.
\end{corollary}

\begin{proof}
A free $A [ \mathbb V / \mathbb V [ 0 ] ]$-module is automatically flat over $A [ \mathbb V / \mathbb V [ 0 ] ]$.
\end{proof}

\begin{theorem}\label{gen_normal}
The scheme $\mathfrak N$ is normal.
\end{theorem}

\begin{proof}
By the Serre criterion and Propositions \ref{LI} and \ref{CM}, it suffices to show that $\mathfrak N$ is integral. The intersection 
$$\mathfrak N \dot{\cap} ( V _1 \oplus \mathbb V _0 \oplus \mathbb V _2 )$$
is integral since $\mathfrak N \dot{\cap} \mathbb V _0$ is so. Let $I _1 = ( P _i )$ and $I _2 := ( x _{i j} = 0 ; i , j > n )$ be ideals of $A [ \mathbb V ]$, where $\{ x _{ij} \} \in \mathrm{Alt} ( 2n ) \cong V _2$. Then, the ideal $I := I _0 + I _1$ is prime. The set of $A$-valued points $T ( A ) \cong ( A ^{\times} ) ^n$ is Zariski dense in $T$. By the Bruhat decomposition, it follows that $G ( A )$ is Zariski dense in $G$. By Lemma \ref{satV}, this implies that
\begin{align*}
I _1 = I _1 + \bigcap _{g \in G ( A )} ( g ^* I _2  ) = &  \bigcap _{g \in G ( A )} ( I _1 + g ^* I _2  ) = \bigcap _{g \in G ( A )} g ^* I \\
 = & \bigcap _{g \in G ( A )} ( g ^* I \otimes _A \mathbb K ) \cap A [ \mathbb V ] = \left< \mathbb K [ \mathbb V ] _+ ^G \right> \cap A [ \mathbb V ].
\end{align*}
We have $( \nu _{\mathbb C} ) _* \mathcal O _{F _{\mathbb C}} \cong \mathbb C [ \mathfrak N ]$ by \cite{K1} 1.2 and the Zariski main theorem. In particular, $\mathbb C [ \mathfrak N ]$ and its subring $\mathbb K [ \mathfrak N ]$ are integral. Therefore, the RHS is an ideal whose quotient does not contain a zero divisor.
\end{proof}

\begin{theorem}\label{moment image}
The image of $\nu$ is equal to $\mathfrak N$.
\end{theorem}

\begin{proof}
Since $F$ is smooth over $A$, the $A$-algebra
$$B := \Gamma ( \mathbb V, \nu _* \mathcal O _{F} ) = \Gamma ( F, \mathcal O _F )$$
is torsion-free over $A$. Hence, $B$ is flat and integral over $A$.\\
We have $\mathbb C \otimes _A B \cong ( \nu _{\mathbb C} ) _* \mathcal O _{F _{\mathbb C}} \cong \mathbb C [ \mathfrak N ]$. By the proof of Theorem \ref{gen_normal}, we have
$$\left< A [ \mathbb V ] ^G _+ \right> = \left< \mathbb C [ \mathbb V ] ^{G _{\mathbb C}} _+ \right> \cap A [ \mathbb V ].$$
In particular, the natural map $A [ \mathbb V ] \ni f \mapsto f. 1 \in B$ factors through $A [ \mathfrak N ]$. Here we have $( \mathrm{Im} \nu ) ( \mathbb C ) = \mathfrak N ( \mathbb C )$ and $( \mathrm{Im} \nu ) ( \Bbbk ) = \mathfrak N ( \Bbbk )$ as sets. This implies that $B$ must be normal since $A [ \mathfrak N ]$ is so. Therefore, we conclude $B = A [ \mathfrak N ]$ as desired.
\end{proof}

\begin{definition}[Marked partitions]
A marked partition ${\bm \lambda} = ( \lambda, a )$ is a partition $\lambda = ( \lambda _1 \ge \lambda _2 \ge \ldots )$ of $n$, together with a sequence $a = ( a _1, a _2, \ldots )$ of integers such that:
\begin{enumerate}
\item $0 \le a _k \le \lambda _k$ for each $k$;
\item $a _k = 0$ if $\lambda _{k + 1} = \lambda _k$;
\item $\lambda _p - \lambda _q > a _p - a _q > 0$ if $p < q$ and $a _p \neq 0 \neq a _q$.
\end{enumerate}
\end{definition}

Let $X = X _1 \oplus X _2 \in \mathfrak N$. Since $V _1 ^* \cong V _1$, we can regard $X _2 \in \mathrm{End} ( V _1 )$ via the embedding
\begin{eqnarray}
X _2 \in V _2 \stackrel{\cong}{\longrightarrow} \mathrm{Alt} ( V _1 ) \subset ( V _1 ) \boxtimes ( V _1 ) \cong \mathrm{End} _A ( V _1 ).\label{Howto}
\end{eqnarray}

\begin{definition}\label{kinv}
Let ${\bm \lambda} = ( \lambda, a )$ be a marked partition of $n$. An element $X \in \mathfrak N$ is said to have $\mathbb K$-invariant ${\bm \lambda}$ if the following conditions hold:
\begin{itemize}
\item $X _2$ is a nilpotent element with its Jordan type $( \lambda _1,\lambda _1, \lambda _2, \lambda _2, \ldots)$;
\item There exists a family of vectors $\{ \xi ( i ) \} _{i \ge 0} \in ( V _1 ) _{\mathbb K}$ such that:
\begin{enumerate}
\item $\xi ( i ) \neq 0$ if and only if $a _i \neq 0$;
\item $X _1 = \sum _j X _2 ^{\lambda _j - a _j} \xi ( j )$;
\item $X _2 ^{\lambda _i - 1} \xi ( i ) \neq 0$ and $X _2 ^{\lambda _i} \xi ( i ) = 0$ for each $i$ such that $\xi ( i ) \neq 0$.
\end{enumerate}
\end{itemize}
We say that $X \in \mathfrak N$ have $\Bbbk$-invariant ${\bm \lambda}$ if the same conditions hold by replacing every $\mathbb K$ by $\Bbbk$. The $\mathbb K$-invariant (resp. the $\Bbbk$-invariant) of $X$ is denoted by ${\bm \lambda} ( X _{\mathbb K} ) = ( \lambda ( X _{\mathbb K} ), a ( X _{\mathbb K} ) )$ (resp. ${\bm \lambda} ( X _{\Bbbk} )$). We define $\mathcal O _{\bm \lambda}$ to be the locally closed subscheme of $\mathfrak N _{\mathbb K}$ whose $\mathbb K$-valued points shares the $\mathbb K$-invariant ${\bm \lambda}$.
\end{definition}

It is standard that the $\mathbb K$-invariants and $\Bbbk$-invariants are invariants under the $G$-action.

\begin{theorem}[\cite{K1} Theorem 1.14]\label{clsn}
Two points $X, Y \in \mathfrak N _{\mathbb K}$ are $G _{\mathbb K}$-conjugate iff $X$ and $Y$ share the same $\mathbb K$-invariant. In particular, $( \mathcal O _{\bm \lambda} ) _{\mathbb K}$ is a single $G _{\mathbb K}$-orbit. \hfill $\Box$
\end{theorem}

\section{A geometric family of nilcones}\label{gfn}
We assume the same setting as in the previous section. The $G _{\Bbbk}$-module $\mathfrak g _{\Bbbk}$ has a non-trivial $B$-eigenvector with its weight $\epsilon _1 + \epsilon _2$. This yields the following short exact sequence of $G _{\Bbbk}$-modules:
\begin{eqnarray}
0 \longrightarrow \left( V _2 \right) _{\Bbbk} \longrightarrow \mathfrak g _{\Bbbk} \longrightarrow \left( V _1 \right) _{\Bbbk} ^{[ 1 ]} \longrightarrow 0.\label{extseq}
\end{eqnarray}
Thus, we have a $G _{\Bbbk}$-equivariant flat deformation
$$\pi _V : \mathcal V \longrightarrow \mathbb A ^1 _{\Bbbk}$$
of $\mathfrak g _{\Bbbk}$ such that $\pi ^{-1} _V ( 0 ) \cong \left( V _1 \right) _{\Bbbk} ^{[1]} \oplus ( V _2 ) _{\Bbbk}$ and $\pi ^{-1} _V ( x ) \cong \mathfrak g _{\Bbbk}$ ($x \neq 0$) as $G _{\Bbbk}$-modules. Let $\mathcal V ^+ \subset \mathcal V$ be its $B _{\Bbbk}$-equivariant smooth subfamily such that $\mathcal V ^+ \cap \pi ^{-1} _V ( 0 ) \cong \left( V _1 ^+ \right) _{\Bbbk} ^{[1]} \oplus ( V _2 ^+ ) _{\Bbbk}$ and $\mathcal V ^+ \cap \pi ^{-1} _V ( x ) \cong \mathfrak n _{\Bbbk}$ ($x \neq 0$).

By (\ref{extseq}) and the isomorphism $\mathsf{Sym} ^2 V _1 \cong \mathfrak g$ over $A$, we have the following $G _{\Bbbk}$-equivariant map
$$ml : \mathbb V _{\Bbbk} = ( V _1 ) _{\Bbbk} \oplus ( V _2 ) _{\Bbbk} \ni ( X _1 \oplus X _2 ) \mapsto \mathsf{Sym} ^2 X _1 + X _2 \in \mathfrak g _{\Bbbk}.$$
This map is $G _{\Bbbk}$-equivariant and finite as a map between affine algebraic varieties. By restriction, we obtain a commutative diagram
\begin{eqnarray}\xymatrix{
\mathbb A ^1 _{\Bbbk} \times \mathbb V ^+ _{\Bbbk} \ar[r] ^{\underline{\mathsf{ml}}} & \mathcal V ^+\\
\{ 0 \} \times \mathbb V ^+ _{\Bbbk} \ar[u] \ar[r] ^{\mathsf{F} _1} & \{ 0 \} \times ( V ^+ _1 ) ^{[1]} _{\Bbbk} \oplus ( V _2 ^+ ) _{\Bbbk} \ar[u]
},\label{positive degeneration}
\end{eqnarray}
where the map $\underline{\mathsf{ml}}$ is the natural prolongization of the map $ml$, and the map $\mathsf{F} _1$ is the product of the Frobenius map of the first component of $\mathbb V ^+ _{\Bbbk}$ and the identity map of the second component of $\mathbb V ^+ _{\Bbbk}$.

Let $\mathcal F := ( G _{\Bbbk} \times \mathbb A ^1 _{\Bbbk} ) \times ^{( B _{\Bbbk} \times \mathbb A ^1 _{\Bbbk} )} \mathcal V ^+$. This is a flat family over $\mathbb A ^1 _{\Bbbk}$. We define $\mathcal N _S := ( G _{\Bbbk} \times \mathbb A ^1 _{\Bbbk} ) \mathcal V ^+ \subset \mathcal V$. We define $\pi := \pi _V \MID _{\mathcal N _S}$. The diagram (\ref{positive degeneration}) gives a $G _{\Bbbk}$-equivariant commutative diagram
\begin{eqnarray}\xymatrix{
\mathbb A ^1 _{\Bbbk} \times F _{\Bbbk} \ar[d] \ar[r] ^{\widetilde{\mathsf{ml}}}& \mathcal F \ar[d] & \{ 0 \} \times F _{\Bbbk} \ar[d] ^{\nu _{\Bbbk}} \ar[l]\\
\mathbb A ^1 _{\Bbbk} \times \mathfrak N _{\Bbbk} \ar[r] ^{\mathsf{ml}}& \mathcal N _S & \{ 0 \} \times \mathfrak N _{\Bbbk} \ar[l]
}.\label{induction to whole space}
\end{eqnarray}
Here the vertical arrows are defined as $( G _{\Bbbk} \times \mathbb A ^1 _{\Bbbk} )$- or $G _{\Bbbk}$-translations of (\ref{positive degeneration}) inside $\mathbb A ^1 _{\Bbbk} \times \mathbb V _{\Bbbk}$, $\mathcal V$, or $\mathbb V _{\Bbbk}$, respectively. Thanks to the surjectivity of $\underline{\mathsf{ml}}$ at (\ref{positive degeneration}) and Theorem \ref{HLcls} and Corollary \ref{satk}, the map $\mathsf{ml}$ at (\ref{induction to whole space}) is surjective at the level of points. Since $\mathcal F$ is flat over $\mathbb A ^1 _{\Bbbk}$, it follows that $\mathcal N _S$ is flat over $\mathbb A ^1 _{\Bbbk}$. Let $\mathsf{m} : \mathfrak N _{\Bbbk} \longrightarrow \mathcal N _{\Bbbk}$ be the map obtained by the specialization of $\mathsf{ml}$ to the fiber at the point $\{ 1 \} \in \mathbb A ^1 _{\Bbbk}$.

\begin{theorem}\label{pcflatex}
Each $G _{\Bbbk}$-orbit $\mathcal O$ of $\mathfrak N _{\Bbbk}$ extends to a flat family of $G _{\Bbbk}$-orbits in $\mathcal N _S$ with its general fiber isomorphic to $\mathsf{m} ( \mathcal O )$. Moreover, this yields a one-to-one correspondence between $G _{\Bbbk}$-orbits of $\mathfrak N _{\Bbbk}$ and $\mathcal N _{\Bbbk}$ which preserves the closure relations.
\end{theorem}

\begin{proof}
The map $\mathsf{m}$ is a $G ( \Bbbk )$-equivariant isomorphism at the level of $\Bbbk$-valued points. Hence, we have an equi-dimensional one-to-one correspondence between the $G _{\Bbbk}$-orbits of $\mathfrak N _{\Bbbk}$ and $\mathcal N _{\Bbbk}$.

We have an equi-dimensional family $\mathbb O _{S} := \mathsf{ml} ( \mathbb A ^1 _{\Bbbk} \times \mathcal O )$ for each $G _{\Bbbk}$-orbit $\mathcal O$. Here the map $\mathsf{F} _1$ is finite. As a result, $\overline{\mathbb O _{S}}$ is an equi-dimensional family such that each fiber contains a unique dense open $G _{\Bbbk}$-orbit.

The number of $G _{\Bbbk}$-orbits of $\mathcal N _{\Bbbk}$ is finite by \cite{Hes}. It is clear that $\overline{\mathbb O _{S}}$ is irreducible. The family $\mathcal N _S$ is locally trivial along $\mathbb A ^1 \backslash \{ 0 \}$. It follows that the open dense $G _{\Bbbk}$-orbit of $\overline{\mathbb O _{S}}$ is constant along $\mathbb A ^1 \backslash \{ 0 \}$.

Conversely, each $G _{\Bbbk}$-orbit $\mathbb O \subset \mathcal N _{\Bbbk}$ determines a $G _{\Bbbk}$-orbit $\mathcal O ^{\prime} \subset \mathsf{F} _1 ( \mathfrak N )$ via
\begin{eqnarray}
\mathcal O ^{\prime} \stackrel{\text{open}}{\hookrightarrow} \pi ^{-1} ( 0 ) \cap \overline{\mathbb O _S ^{\prime}} \subset \overline{\mathbb O _S ^{\prime}},\label{cinv}
\end{eqnarray}
where $\mathbb O _S ^{\prime}$ is a family of single $G _{\Bbbk}$-orbits $\mathbb O$ in $\mathcal N _S$ along $\mathbb A ^1 \backslash \{ 0 \}$. Here $\mathcal O^{\prime}$ is uniquely determined since $\overline{\mathbb O _S ^{\prime}} = \overline{\mathsf{ml} ( \mathbb A ^1 \times \mathcal O ^{\prime\prime} )}$ for some $G _{\Bbbk}$-orbit $\mathcal O ^{\prime\prime}$ of $\mathfrak N _{\Bbbk}$.

Since $\mathcal O ^{\prime} = \mathsf{F} _1 ( \mathcal O ^{\prime\prime} )$, this establishes a one-to-one correspondences between $G _{\Bbbk}$-orbits in $\mathfrak N _{\Bbbk}$, that of $\mathcal N _{\Bbbk}$, and families of single $G _{\Bbbk}$-orbits of $\mathcal N_S$, which is compatible with the bijection induced by $\mathsf{m}$. Now the closure relation of $\mathfrak N _{\Bbbk}$ extends to the closure relation of $\mathbb O _S$, which guarantees that the closure relation is preserved by $\mathsf{m}$.

Since $\mathbb A ^1 _{\Bbbk}$ is one-dimensional, $\overline{\mathbb O _S}$ yields the required flat family. 
\end{proof}

In the below, we denote the one-to-one correspondence from the set of $G _{\Bbbk}$-orbits of $\mathfrak N _{\Bbbk}$ to that of $\mathcal N _{\Bbbk}$ described in Theorem \ref{pcflatex} by $\mathsf{df}$.

\begin{corollary}
The numbers of $G _{\Bbbk}$-orbits in $\mathcal N _{\Bbbk}$ and $\mathfrak N _{\Bbbk}$ are equal to the number of $G _{\mathbb K}$-orbits in $\mathfrak N _{\mathbb K}$.
\end{corollary}

\begin{proof}
By Spaltenstein \cite{Spa} 3.9, we deduce that the number of $G _{\Bbbk}$-orbits in $\mathcal N _{\Bbbk}$ is equal to the set of bi-partitions of $n$. Hence, Theorem \ref{clsn} implies the result (see also Theorem \ref{mp<->bp}).
\end{proof}

\begin{corollary}\label{kcomplete}
Two $G _{\Bbbk}$-orbits in $\mathfrak N _{\Bbbk}$ are equal if and only if their $\Bbbk$-invariants are equal. \hfill $\Box$
\end{corollary}

\section{Special elements of Weyl groups}\label{seW}

Keep the setting of the previous section. In this section, we always work over an algebraically closed field and drop the specific field from the notation.

\begin{theorem}[\cite{K1} Theorem 1.14 and Proof of Proposition 1.16]\label{mp<->bp}
Let ${\bm \lambda} = ( \lambda, a )$ be a marked partition of $n$. We define a sequence of integers $b = ( b _1, b _2, \ldots )$ as
$$b _i := \begin{cases} a _i & (a _i \neq 0)\\ \max \left\{ \{ a _j + \lambda _i - \lambda _j ; j < i \} \cup \{ a _j ; j \ge i\} \right\} & (a _i = 0)\end{cases}.$$
Then, we define two partitions $\mu ^{{\bm \lambda}}$ and $\nu ^{{\bm \lambda}}$ as
$$\mu ^{{\bm \lambda}} _i = b _i, \nu ^{{\bm \lambda}} _i = \lambda _i - b _i.$$
The pair $( \mu ^{{\bm \lambda}}, \nu ^{{\bm \lambda}} )$ gives a bi-partition of $n$. Moreover, this assignment establishes a bijection between the set of marked partitions of $n$ and the set of bi-partitions of $n$. \hfill $\Box$
\end{theorem}

We refer the bi-partition $( \mu ^{\bm \lambda}, \nu ^{\bm \lambda} )$ the associated bi-partition of a marked partition $\bm \lambda$. Let $\mathcal O _{\bm \lambda}$ be the $G$-orbit with its $\mathbb K$-invariant (or $\Bbbk$-invariant, depending on the base field) ${\bm \lambda}$.

\begin{definition}[Special elements]\label{se}
Let ${\bm \lambda}$ be a marked partition and let $( \mu, \nu ) := ( \mu ^{{\bm \lambda}}, \nu ^{{\bm \lambda}} )$ be its associated bi-partition. We define an element $w _{\bm \lambda} \in W$ as
$$w _{\bm \lambda} \epsilon _i = \begin{cases} \epsilon _{n - m + 1} & (i = ( {} ^t \mu ) _m ^{\ge})\\ - \epsilon _{\left| \nu \right| + ( {} ^t \mu ) ^< _m + i - ( {} ^t \mu ) _m ^{>} - m + 1} & (( {} ^t \mu ) _m ^{>} < i < ( {} ^t \mu ) _m ^{\ge} ) \\ - \epsilon _{( {} ^t \nu ) ^> _m + i - ( {} ^t \nu ) ^< _m - \left| \mu \right|} & (\left| \mu \right| + ( {} ^t \nu ) ^{<} _m < i \le \left| \mu \right| + ( {} ^t \nu ) ^{\le} _m ) \end{cases},$$
where $m$ is some natural number. We put $\mathbb V ^{{\bm \lambda}} := \mathbb V ^+ \cap {} ^{w _{\bm \lambda} ^{-1}} \mathbb V ^+$. We define a sequence of integers $d ^{\bm \lambda} := \{ d _i ^{\bm \lambda }\} _{i = 0} ^{\mu _1 + \nu _1}$ as $d _0 ^{{\bm \lambda}} := 0$ and
$$d _1 ^{{\bm \lambda}} := ( {} ^t \mu ) ^{\ge} _{\mu _1}, \ldots, d _{\mu _1} ^{{\bm \lambda}} := ( {} ^t \mu ) ^{\ge} _{1}, d _{\mu _1 + 1} ^{{\bm \lambda}} := \left| \mu \right| +  ( {} ^t \nu ) ^{\le} _{1}, \ldots, d _{\mu _1 + \nu _1} ^{{\bm \lambda}} := \left| \mu \right| + ( {} ^t \nu ) ^{\le} _{\nu _1}.$$
We may drop the superscript ${} ^{\bm \lambda}$ if the meaning is clear from the context.
\end{definition}

\begin{lemma}\label{wdlambda}
Keep the setting of Definition \ref{se}.
\begin{enumerate}
\item We have $\epsilon _i - \epsilon _j \not\in \Psi ( \mathbb V ^{{\bm \lambda}} )$ if and only if $i \ge j$ or one of the following conditions hold for some natural number $m$:
\begin{enumerate}
\item $( {} ^t \mu ) _m ^{>} < i, j < ( {} ^t \mu ) _m ^{\ge}$;
\item $\left| \mu \right| + ( {} ^t \nu ) _m ^{<} < i, j \le \left| \mu \right| + ( {} ^t \nu ) _m ^{\le}$;
\item $j = ( {}^t \mu ) ^{\ge} _m$, $1 \le i \le \left| \mu \right|$, and $i \not\in \{ ( {} ^t \mu ) ^{\ge} _l \} _{l > m}$.
\end{enumerate}
\item We have $\epsilon _i + \epsilon _j \in \Psi ( \mathbb V ^{{\bm \lambda}} )$ if and only if $i \neq j$ and $i, j \in \{ ( {}^t \mu ) ^{\ge} _m \} _m$;
\item We have $\epsilon _i \in \Psi ( \mathbb V ^{{\bm \lambda}} )$ if and only if $i \in \{ ( {}^t \mu ) ^{\ge} _m \} _m$.
\end{enumerate}
\end{lemma}

\begin{proof}
Straight-forward.
\end{proof}

For each $\alpha \in R ^+$, we denote the corresponding unipotent one-parameter subgroup of $G$ by $U _{\alpha} \subset N$.

\begin{lemma}\label{V01}
Keep the setting of Definition \ref{se}. We have
$$\mathbb V ^{{\bm \lambda}} \subset \overline{B ( \mathbb V ^{{\bm \lambda}} \cap ( \mathbb V _0 \oplus \mathbb V _1 ) )},$$
where $\mathbb V _0 \subset V _2$ and $\mathbb V _1 \subset V _1$ are $G _0$-stable subspaces defined at (\ref{edec}).
\end{lemma}

\begin{proof}
We put $\gamma _m := \epsilon _{( {} ^t \mu ) ^{\ge} _m}$. By Lemma \ref{wdlambda}, we have $\epsilon _i + \epsilon _j \in \Psi ( \mathbb V ^{{\bm \lambda}} )$ only if $i, j \in \{ ( {}^t \mu ) ^{\ge} _m \} _m$. Moreover, $i < j$ and $\epsilon _i + \epsilon _j \in \Psi ( \mathbb V ^{{\bm \lambda}} )$ implies $\epsilon _i - \epsilon _j \in \Psi ( \mathbb V ^{{\bm \lambda}} )$. We put
$$U _+ := \prod _{l \le m} U _{\gamma _l + \gamma _m}, \mathbb V _0 ^{\circ} := \bigoplus _{l > m} \mathbb V [ \gamma _l - \gamma _m ], \text{ and } \mathbb V _2 ^{\circ} := \bigoplus _{l < m} \mathbb V [ \gamma _l + \gamma _m ].$$
We have $\mathbb V _0 ^{\circ}, \mathbb V _2 ^{\circ} \subset \mathbb V ^{{\bm \lambda}}$. By a weight comparison, we deduce
$$U _+ \mathbb V ^{{\bm \lambda}} = U _+ ( ( \mathbb V _0 ^{\circ} + \mathbb V _2 ^{\circ} ) \cap \mathbb V ^{{\bm \lambda}} ) + \mathbb V ^{{\bm \lambda}}.$$
Here $U _+ \mathbb V _0 ^{\circ} \subset \mathbb V _0 ^{\circ} + \mathbb V _2 ^{\circ}$ is a dense open subset. Thus, we conclude
$$\mathbb V ^{{\bm \lambda}} \subset \overline{U _+ ( \mathbb V ^{{\bm \lambda}} \cap ( \mathbb V _0 \oplus \mathbb V _1 ) )} \subset \overline{B ( \mathbb V ^{{\bm \lambda}} \cap ( \mathbb V _0 \oplus \mathbb V _1 ) )}$$
as desired.
\end{proof}

\begin{lemma}\label{basic identity}
Keep the setting of Definition \ref{se}. We define
$$\mathbb V ^{{\bm \lambda}} _{01} := \bigoplus _{i \le d _{\mu _1}} \mathbb V [ \epsilon _i ] \oplus \bigoplus _{l < m} \bigoplus _{\scriptsize \begin{matrix} d _l < i \le d _{l + 1}\\ d _m < j \le d _{m + 1}\end{matrix}} \mathbb V [ \epsilon _i - \epsilon _j].$$
Then, we have $\overline{B \mathbb V ^{{\bm \lambda}} _{01}} = \overline{B \mathbb V ^{{\bm \lambda}}}$.
\end{lemma}

\begin{proof}
We put $\Psi := \Psi ( \mathbb V ^{{\bm \lambda}} \cap ( \mathbb V _0 \oplus \mathbb V _1 ) )$. Since $\mathbb V ^{{\bm \lambda}} \cap ( \mathbb V _0 \oplus \mathbb V _1 ) \subset \mathbb V ^{{\bm \lambda}} _{01}$, it suffices to prove the inclusion
$$\mathbb V ^{{\bm \lambda}} _{01} \subset \overline{T N _0 ( \mathbb V ^{{\bm \lambda}} \cap ( \mathbb V _0 \oplus \mathbb V _1 ) )} \subset \overline{B \mathbb V ^{{\bm \lambda}}}.$$
Here the second inclusion is obvious. We have
$$\Psi ( \mathbb V ^{{\bm \lambda}} _{01} ) \backslash \Psi = \{ \epsilon _i ; i \in [ 1, d _{\mu _1} ] \backslash \{ ( {} ^t \mu ) _m ^{\ge} \} _m \} \cup \{ \epsilon _i - \epsilon _j ; i < j, i \not\in \{ ( {} ^t \mu ) _m ^{\ge} \} _m \ni j \}.$$
We deduce that
$$\left( \prod _{i < \left| \mu \right|} U _{\epsilon _i - \epsilon _{\left| \mu \right|}} \right) \mathbb V [ \epsilon _{\left| \mu \right|} ] \subset \mathbb V ^{{\bm \lambda}} _{01} \cap \mathbb V _1$$
is a dense open subset by the comparison of weights. We put
$$U _- := \prod _{i < j ; j \in \{ ( {} ^t \mu ) _m ^{\ge} \} _{m}} U _{\epsilon _i - \epsilon _{j}} \subset N _0.$$
It is easy to see that $U _-$ does not depend on the order of the product. By the comparison of weights, we have a dense open subset
$$U _- ( \mathbb V ^{{\bm \lambda}} \cap \mathbb V _0 ) \subset \mathbb V ^{{\bm \lambda}} _{01} \cap \mathbb V _0.$$
Applying the $T$-action, we conclude that the first inclusion is dense.
\end{proof}

In the below, we denote by $\mathbb V ^{\bm \lambda} _i$ ($i = 0, 1$) the spaces $\mathbb V ^{\bm \lambda} _{01} \cap \mathbb V _i$ coming from the statement of Lemma \ref{basic identity} for each marked partition $\bm \lambda$.

\begin{proposition}\label{base_set}
Let ${\bm \lambda}$ be a marked partition. We have $\overline{G \mathbb V ^{{\bm \lambda}}} = \overline{\mathcal O _{{\bm \lambda}}}$.
\end{proposition}

\begin{proof}
By Definition \ref{kinv} and Lemma \ref{wdlambda}, we deduce that ${\mathcal O} _{{\bm \lambda}} \subset G \mathbb V ^{{\bm \lambda}}$. Thus, it suffices to check $\mathbb V ^{{\bm \lambda}} _{01} \subset \overline{\mathcal O _{{\bm \lambda}}}$. We define an increasing filtration
$$\{ 0 \} = F _0 \subsetneq F _1 \subsetneq \cdots \subsetneq F _{\mu _1 + \nu _1 - 1} \subsetneq F _{\mu _1 + \nu _1} = V _1 ^+ := \mathbb V ^+ \cap V _1$$
as $F _k := \bigoplus _{i \le d _k} \mathbb V [ \epsilon _i ]$. By a weight comparison, each $x \in \mathbb V ^{{\bm \lambda}} _0$ preserves the flag $\{ F _k \} _k$ when regarded as an element of $\mathrm{End} ( V _1 )$ as in (\ref{Howto}). Moreover, the set of elements $x$ in $\mathbb V ^{{\bm \lambda}} _0$ which satisfies
$$\dim x ^l F _k /  F _{k - l - 1} = \min \{ \dim F _{k - m + 1} / F _{k - m} \} _{m = 1} ^{l} \text{ for every } l$$
is dense in $\mathbb V ^{{\bm \lambda}} _0$. Let $\xi \in F _{\mu _1} \cap V _1 ^+$ be an element such that there exists $\{ \xi ( i ) \} _{i} \subset V _1 ^+$ which satisfies
$$\xi = \sum _{i \ge 1} x ^{\lambda _i - b _i} \xi ( i ), \text{ and } x ^{\lambda _i - 1} \xi ( i ) \neq 0 = x ^{\lambda _i} \xi ( i ) \text{ if } \xi ( i ) \neq 0.$$
Under the above choice of $x$, this condition is an open condition. We rearrange $\{\xi ( i )\}$ according to the following rules: If $b _i = a _j$ for some $j < i$, then we rearrange $\xi ( i ), \xi ( j )$ as $0, x ^{\lambda _j - \lambda _i} \xi ( i ) + \xi ( j )$, and let others unchanged. If $b _i = a _j - \lambda _j + \lambda _i$ for $i > j$, then we rearrange $\xi ( i ), \xi ( j )$ with $0, \xi ( i ) + \xi ( j )$, and let others unchanged.\\
We repeat this procedure for all possible pairs $(i, j)$. We conclude that if we have $\xi ( i ) \neq 0$ for every $i$ such that $a _i \neq 0$ after applying these procedures, then $\xi \oplus x$ has $\mathbb K$-invariant $( \lambda, a )$. Since the former is an open condition on $\xi$, we conclude the result.
\end{proof}

\begin{corollary}\label{dim_change}
Under the setting of Proposition \ref{base_set}, we have
$$\dim \mathcal O _{{\bm \lambda}} = \dim \mathcal O _{( \lambda, \mathbf 0 )} + 2 \left| \mu \right|.$$
\end{corollary}

\begin{proof}
We retain the setting of the proof of Proposition \ref{base_set}. Let $\xi \in V _1 ^+$. Then, we have $\xi \oplus x \in \overline{\mathcal O _{{\bm \lambda}}}$ if and only if $\xi \in \bigoplus _{i \le \left| \mu \right|} V _1 [ \epsilon _i ]$. The Jordan type of $x$ is $\lambda$ (unchanged) if we regard $x \in \mathrm{End} ( V _1 )$ as either $x \in \mathrm{End} ( \mathbb V _1 )$ or $x \in \mathrm{End} ( \mathbb V _{-1} )$. Therefore, the fiber of the projection $\mathcal O _{{\bm \lambda}} \rightarrow \mathcal O _{( \lambda, \mathbf 0 )}$ has dimension $2 \dim F _{\mu _1} = 2 \left| \mu \right|$ as desired.
\end{proof}

\section{Normality of nilcones in characteristic two}\label{nnt}

We assume the setting of \S \ref{gfn}.

\begin{theorem}\label{flat_prolong_A}
For each $G _{\Bbbk}$-orbit $\mathcal O _{\Bbbk}$ of $\mathfrak N _{\Bbbk}$, there exists a $G$-orbit $\mathcal O$ of $\mathfrak N$ such that $\mathcal O \otimes \Bbbk = \mathcal O _{\Bbbk}$.
\end{theorem}

\begin{proof}
Assume that $\mathcal O _{\Bbbk}$ has $\Bbbk$-invariant $\bm \lambda$. Let $\mathcal O _{\mathbb K}$ be a $G_{\mathbb K}$-orbit of $\mathfrak N _{\mathbb K}$ with $\mathbb K$-invariant $\bm \lambda$. By Proposition \ref{base_set} and Lemma \ref{basic identity}, we have $\overline{\mathcal O _{\mathbb K}} = \overline{G \mathbb V ^{\bm \lambda} _{01}} \subset \mathbb V$ over $A$. Since $\overline{\mathcal O _{\mathbb K}}$ is irreducible and dominates $\mathrm{Spec} A$, it is flat over $A$. Therefore, $\overline{\mathcal O _{\mathbb K}}$ is equi-dimensional over $A$. We have $B \mathbb V_{01} ^{\bm \lambda} = G _2 \mathbb V _{01} ^{\bm \lambda}$. For some (and in fact generic) $X \in \mathbb V _{01} ^{\bm \lambda} ( A )$ such that $X _{\mathbb K}$ has $\mathbb K$-invariant $\bm \lambda$ and $X _{\Bbbk}$ has $\Bbbk$-invariant $\bm \lambda$, we have $( \overline{\mathsf{Stab} _{ ( G _2 ) _{\mathbb K}} X _{\mathbb K}} ) _{\Bbbk} = \mathsf{Stab} _{( G _2 ) _{\Bbbk}} X _{\Bbbk}$ by identifying defining equations, which are all linear. Applying the $G _0$-action, we deduce that $B \mathbb V _{01} ^{\bm \lambda}$ contains a dense open $\mathbb G _m$-stable subvariety, which is flat over $A$. Therefore, we conclude $\overline{B \mathbb V _{01} ^{\bm \lambda}} \otimes \Bbbk = \overline{B _{\Bbbk} ( \mathbb V _{01} ^{\bm \lambda} ) _{\Bbbk}}$ via the equality of the leading terms of Hilbert polynomials of $\Bbbk [ \mathbb V ^+ _{\Bbbk} ]$-modules. By Proposition \ref{base_set}, we have $\overline{\mathcal O _{\mathbb K}} \otimes \Bbbk = \overline{\mathcal O} _{\Bbbk}$ and the latter is irreducible. Hence, setting
$$\mathcal O := \overline{\mathcal O _{\mathbb K}} - \bigcup _{\mathcal O ^{\prime}: G _{\mathbb K}\text{-orbit}; \overline{\mathcal O ^{\prime}} \subsetneq \overline{\mathcal O}} \overline{\mathcal O ^{\prime} _{\mathbb K}}$$
yields the result.
\end{proof}

\begin{corollary}\label{normal2en}
The variety $\mathfrak N _{\Bbbk}$ is normal.
\end{corollary}

\begin{proof}
Since $\mathfrak N$ is Cohen-Macauley and flat over $A$, it suffices to prove that $\mathfrak N _{\Bbbk}$ is regular in codimension one. By Theorem \ref{flat_prolong_A}, it follows that every non-dense $G _{\Bbbk}$-orbit in $\mathfrak N _{\Bbbk}$ is codimension at least two. Therefore, we deduce the result.
\end{proof}

\begin{corollary}
The map $\nu _{\Bbbk}$ is strictly semi-small with respect to the stratification given by $G _{\Bbbk}$-orbits.
\end{corollary}

\begin{proof}
For each $G$-orbit $\mathcal O$ of $\mathfrak N$, we choose $X \in \mathcal O ( A )$. Then, the upper-semicontinuity of fiber dimensions of $\nu ^{-1} ( X )$ along $A$ implies the strictness of $\nu _{\Bbbk}$ from that of $\nu _{\mathbb K}$.
\end{proof}

%
%\begin{theorem}\label{clo_rel}
%Let $X = X _1 \oplus X _2 \in \mathfrak N$ be a point with $\Bbbk$-invariant $( \lambda, a )$. Let $\mathcal O _{\mathbb K} \subset \mathfrak N _{\mathbb K}$ be a $G _{\mathbb K}$-orbit with $\mathbb K$-invariant $( \lambda, a )$. Then, there exists a $G _{\mathbb K}$-orbit $\mathcal O ^{\prime} _{\mathbb K}$ such that $X _{\mathbb K} \in \mathcal O ^{\prime} _{\mathbb K}$ and $\mathcal O _{\mathbb K} \subset \overline{\mathcal O ^{\prime}} _{\mathbb K}$ holds.
%\end{theorem}

%\begin{proof}
%Assume to the contrary to deduce contradiction. If we have $\mathcal O _{\mathbb K} \not\subset \overline{\mathcal O ^{\prime}} _{\mathbb K}$, then the dimension of $\mathcal O _{\mathbb K}$
%\end{proof}

\begin{theorem}
The variety $\mathcal N _{\Bbbk}$ is normal.
\end{theorem}

\begin{proof}
The isogeny $ml : \mathbb V \mapsto \mathfrak g$ induces an injective map $\Bbbk [ \mathfrak g _{\Bbbk} ] \hookrightarrow \Bbbk [ \mathbb V _{\Bbbk} ]$. Since this map is $G ( \Bbbk )$-equivariant, we have an inclusion $\Bbbk [ \mathfrak g _{\Bbbk} ] ^G \hookrightarrow \Bbbk [ \mathbb V _{\Bbbk} ] ^G$. The variety $\mathcal N _{\Bbbk}$ is defined from $\mathcal N$ by reduction modulo $2$. Hence, its defining equations are coming from $A [ \mathfrak g _{A} ] ^G$, which are given by polynomials of degree $2, 4, \ldots, 2 n$ (c.f. \cite{B}). In this proof, we understand that $\epsilon _{-i} = - \epsilon _i$ for every $1 \le i \le n$. We set $\{ a _{ij} \} _{ij} \in \mathrm{Alt} ( V _1 ) _{\Bbbk} \cong ( V _2 ) _{\Bbbk}$ as the coordinates with respect to the $T$-eigenbasis so that $a _{ij}$ has eigenvalue $\epsilon _i + \epsilon _j$. Let $\{ v _i \} _i \in ( V _1 ) _{\Bbbk}$ be the $T$-eigenbasis such that $v _i$ is of weight $\epsilon _i$. Let $c _{ij}$ be the $T$-eigenbasis of $\mathfrak g _{\Bbbk}$ of weight $\epsilon _i + \epsilon _j$. We set $c _{i} := c _{ii}$ and $\bar{c} _{i} = c _{-i,-i}$. We have
$$ml ^* ( c _{ij} ) = a _{ij} + v_{i} v_{j} \hskip 2mm (i \neq j) \hskip 2mm \text{ or } v _{i} ^2 \hskip 2mm (i = j).$$
Here we have $P _k ( a_{ij} ^2 ) \in \Bbbk [ \mathfrak g _{\Bbbk} ] \cap \Bbbk [ \mathbb V _{\Bbbk} ] ^{G}$. By explicit calculation, we have
$$P _k ( a_{ij} ^2 ) = \sum _{I \subset [ 1, n ]; \# I = k} \prod _{i \in I} c _i \bar{c} _i + \text{ lower terms with respect to } \{ c _i \}.$$
It follows that the differentials of $P _k ( a_{ij} ^2 )$ with respect to $c _1, \ldots, c _n$ defines a collection of linear independent differentials along the generic point, regardless the values of $a _{ij}$.\\
By degree counting using the inclusions $\Bbbk [ \mathbb V ^{[1]} _{\Bbbk} ] \subset \Bbbk [ \mathfrak g _{\Bbbk} ] \subset \Bbbk [ \mathbb V _{\Bbbk} ]$, these are precisely the defining equations of $\mathcal N _{\Bbbk}$ embedded into $\Bbbk [ \mathbb V _{\Bbbk} ]$. The union of non-dense orbit of $\mathcal N _{\Bbbk}$ has codimension two (c.f. \cite{Hes} or Theorem \ref{pcflatex}). Therefore, the defining equations of $\mathcal N _{\Bbbk}$ defines a set of linearly independent differentials up to codimension two. In conclusion, the same proof as the normality of $\mathfrak N _{\Bbbk}$ implies the result.
\end{proof}

\begin{remark}
The normality of the nilpotent cone of $\mathfrak{sp} ( 2n )$ over a field of characteristic $\neq 2$ is well-known. (See eg. Brion-Kumar \cite{BK} \S 5.)
\end{remark}

\section{Exotic orbital varieties: statement}\label{eovs}
In this section, the term ``flat" means that the object is a flat scheme over $\mathrm{Spec} A$.

Let $\mathcal O$ be a $G$-orbit of $\mathfrak N$. We denote the set of irreducible components of $\mathcal O \cap \mathbb V ^+$  by $\mathrm{Comp} ( \mathcal O )$. We define $\mathrm{Comp} ( \mathbb O )$ for a $G$-orbit of $\mathbb O \subset \mathcal N$ by replacing $\mathbb V ^+$ with $\mathfrak n$.

Similarly, we denote the set of irreducible components of $\mathcal O _{*} \cap \mathbb V ^+ _{*}$ (or $\mathbb O _* \cap \mathfrak n _*$) by $\mathrm{Comp} ( \mathcal O _* )$ or $\mathrm{Comp} ( \mathbb O _* )$ for $* = \mathbb C, \mathbb K, \Bbbk$.

An element of $\mathrm{Comp} ( \mathcal O )$ or $\mathrm{Comp} ( \mathbb O )$ is called an orbital variety.

\begin{theorem}\label{equigen}
Let $\mathcal O _{\mathbb K}$ be a $G _{\mathbb K}$-orbit. Let $\mathfrak X _{\mathbb K} \in \mathrm{Comp} ( \mathcal O _{\mathbb K} )$. Then, we have
\begin{enumerate}
\item $\dim \mathfrak X _{\mathbb K} = \frac{1}{2} \dim \mathcal O _{\mathbb K}$;
\item There exists $w \in W$ such that $\overline{\mathfrak X _{\mathbb K}} = \overline{B _{\mathbb K} ( \mathbb V ^+ _{\mathbb K} \cap {} ^w \mathbb V ^+ _{\mathbb K})}$.
\end{enumerate}
Moreover, the same statements hold when we replace $\mathbb K$ with $\Bbbk$.
\end{theorem}

\begin{proof}
Postponed to \S \ref{eovp}.
\end{proof}

\begin{remark}
{\bf 1)} Theorem \ref{equigen} is an ``exotic" analogue of Joseph's version of the Steinberg-Spaltenstein theorem.\\
{\bf 2)} If the variety $G \times ^B \mathbb V ^+ _{\mathbb C}$ or $\mathcal O _{\mathbb C}$ admits a symplectic structure, then Theorem \ref{equigen} follows from Kaledin \cite{Kal} or \cite{CG}. However, there exists no $G$-invariant holomorphic symplectic form on both of them. We do not know whether it exists when we drop the invariance.
\end{remark}

Let $Z := F \times _{\mathfrak N} F$ and let $\mathsf{p} : Z \rightarrow \mathfrak N$ be the projection. The following result is a straight-forward generalization of \cite{K1} \S 1.2 or \cite{Ste}:

\begin{theorem}[Steinberg] The variety $Z$ is flat of relative dimension $\dim \mathfrak N$. Moreover, it consists of $\# W$ irreducible components. \hfill $\Box$
\end{theorem}

\begin{lemma}\label{ordorb}
Let $\mathbb O$ be a $G$-orbit in $\mathcal N$. Then, $\mathbb O _{\Bbbk}$ is a union of a single $G _{\Bbbk}$-orbit $\mathbb O _{\Bbbk} ^{\prime}$ of dimension $\dim \mathbb O$ and $G _{\Bbbk}$-orbits of dimension $< \dim \mathbb O$.
\end{lemma}

\begin{proof}
Consider the natural embedding $\iota : \mathcal N _{\mathbb K} \subset \mathrm{Mat} ( 2n ) _{\mathbb K}$. It is well-known that the induced map $G _{\mathbb K} \backslash \mathcal N _{\mathbb K} \hookrightarrow \mathop{GL} ( 2n ) _{\mathbb K} \backslash \mathrm{Mat} ( 2n ) _{\mathbb K}$ is injective. (See eg. Tanisaki \cite{T} P152 for this kind of phenomenon.) Hence, $\mathbb O _{\Bbbk}$ is a union of $G _{\Bbbk}$-orbits with the same Jordan normal form (in $\mathrm{Mat} ( 2n, \Bbbk )$.) By Hesselink \cite{Hes}, the maximal dimension of $G _{\Bbbk}$-orbits in $\mathbb O _{\Bbbk}$ is attained by a unique orbit as desired.
\end{proof}

\begin{definition}\label{free}
Let $X \in \mathfrak N$. Then, we define a subscheme
$$\mathcal G _X := \{ g \in G ; X \in g \mathbb V ^+ \} \subset G.$$
It is clear that $( \mathcal G _X ) _{\mathbb K}$ admits a free left $\mathrm{Stab} _G ( X )_{\mathbb K}$-action and a free right $B _{\mathbb K}$-action. The same statement holds if we replace $\mathbb K$ with $\Bbbk$.
\end{definition}

Let $\mathcal O$ be a $G$-orbit of $\mathfrak N$. For each $X \in \mathcal O$, we define
$$\mathcal E _X := \mathcal G _X / B \subset \mathcal B$$
and call it the (exotic) Springer fiber along $X$. By taking conjugation, we know that $( \mathcal E _X ) _\mathbb K \cong ( \mathcal E _Y ) _\mathbb K$ and $( \mathcal E _X ) _\Bbbk \cong ( \mathcal E _Y ) _\Bbbk$ hold if $X, Y \in \mathcal O ( A )$.

\begin{lemma}\label{oc_A}
Keep the setting of Definition \ref{free}. Let $\mathcal O$ be a $G$-orbit such that $X \in \mathcal O ( \mathbb K )$. Let $\{ \mathcal G _X ^i \} _i$ be the set of irreducible components of $( \mathcal G _X ) _{\mathbb K}$. Then, the assignment
$$\mathrm{Comp} ( \mathcal O _{\mathbb K} ) \ni \mathcal G _X ^i X \mapsto \mathcal G _X ^i / B _{\mathbb K} \subset ( \mathcal E _X ) _\mathbb K$$
establishes one-to-one correspondences between the sets of irreducible components of $\mathcal O _{\mathbb K} \cap \mathbb V ^+ _{\mathbb K}, ( \mathcal G _X ) _\mathbb K$, and $( \mathcal E _X )_{\mathbb K}$. The same statement holds if we replace $\mathbb K$ with $\Bbbk$. 
\end{lemma}

\begin{proof}
The assignments $\mathcal G _X ^i \mapsto \mathcal G _X ^i / \mathrm{Stab} _G ( X ) _{\mathbb K} \cong \mathcal G _X ^i X \in \mathrm{Comp} ( \mathcal O _{\mathbb K} )$ and $\mathcal G _X ^i \mapsto \mathcal G _X ^i / B _{\mathbb K}$ gives a surjection from the set of irreducible components of $( \mathcal G _X ) _{\mathbb K}$ and the other two sets. Hence, these assignments fail to be bijective only if $\mathrm{Stab} _G ( X ) _{\mathbb K}$ or $B _{\mathbb K}$ is not connected. The group $\mathrm{Stab} _G ( X ) _{\mathbb K}$ is connected by \cite{K1} Proposition 4.5. The Borel subgroup $B _{\mathbb K} \subset G _{\mathbb K}$ is clearly connected. Entirely the same proof works for $\Bbbk$ (including \cite{K1} Proposition 4.5).
\end{proof}

Let $N ( \mathfrak O )$ be the number of orbital varieties attached to some orbit $\mathfrak O$. In the rest of this section, we assume
\begin{itemize}
\item[$(\spadesuit)_1$] We have $\sum _{\mathcal O _{\mathbb K} \in G _{\mathbb K} \backslash \mathfrak N _{\mathbb K}} N ( \mathcal O _{\mathbb K} ) ^2 = \# W$.
\item[$(\spadesuit)_2$] There are at least $N ( \mathcal O _{\mathbb K} )$ orbital varieties of $\mathcal O _{\Bbbk}$ contained in the closures of $\mathrm{Comp} ( \mathcal O _{\mathbb K} )$ for every $G$-orbit of $\mathfrak N$.
\end{itemize}
The statements $(\spadesuit) _1$, $( \spadesuit) _2$ themselves are proved at Corollary \ref{PW} and Corollary \ref{Jind}, respectively. (The results presented in the rest of this section is used only in the proof of Theorem \ref{Josmain}.)

\begin{proposition}\label{fdef}
Let $\mathcal O$ be a $G$-orbit of $\mathfrak N$. Then, every element of $\mathrm{Comp} ( \mathcal O )$ is flat with relative dimension $\frac{1}{2} \dim \mathcal O _{\mathbb K}$.
\end{proposition}

\begin{proof}
Let $\mathfrak Y \in \mathrm{Comp} ( \mathcal O )$. If $\mathfrak Y$ dominates $\mathrm{Spec} A$, then it is flat since $A$ is one-dimensional. Hence, we assume that $\mathfrak Y ( \mathbb K ) = \emptyset$ in order to deduce contradiction. By Theorem \ref{equigen}, we have $\mathfrak Y ( \Bbbk )  \subset \mathfrak X _{\Bbbk} ( \Bbbk )$ for some $\mathfrak X _{\Bbbk} \in \mathrm{Comp} ( \mathcal O _{\Bbbk})$. Since $\mathcal O _{\Bbbk}$ is an single $G _{\Bbbk}$-orbit, we can further assume $\mathfrak Y = \mathfrak X _{\Bbbk}$. Here $\mathfrak Y$ does not appear in the specializations of the closure of $\mathrm{Comp} ( \mathcal O _{\mathbb K} )$. It follows that $N ( \mathcal O _{\Bbbk} ) > N ( \mathcal O _{\mathbb K} )$. Hence, we have
\begin{eqnarray}
\sum _{\mathcal O _{\Bbbk} \in G _{\Bbbk} \backslash \mathfrak N _{\Bbbk}} N ( \mathcal O _{\Bbbk} ) ^2 > \sum _{\mathcal O _{\mathbb K} \in G _{\mathbb K} \backslash \mathfrak N _{\mathbb K}} N ( \mathcal O _{\mathbb K} ) ^2 = \# W.\label{compest}
\end{eqnarray}
For $X \in \mathcal O _{\Bbbk}$, we have $\mathsf{p} ^{-1} _{\Bbbk} ( X ) \cong \nu _{\Bbbk} ^{-1} ( X ) \times \nu _{\Bbbk} ^{-1} ( X )$. We have
$$\dim \mathsf{p} ^{-1} _{\Bbbk} ( \mathcal O _{\Bbbk} ) = 2 \dim \nu _{\Bbbk} ^{-1} ( X ) + \dim \mathcal O _{\Bbbk} = \dim \mathfrak N _{\Bbbk} = \dim Z _{\Bbbk}.$$
By Theorem \ref{equigen} 1) and Lemma \ref{oc_A}, $\overline{\mathsf{p} ^{-1} _{\Bbbk} ( \mathcal O _{\Bbbk} )}$ is a union of $N ( \mathcal O _{\Bbbk} ) ^2$ irreducible components of $Z _{\Bbbk}$. By (\ref{compest}), $Z _{\Bbbk}$ has more than $\# W$ irreducible components. This contradicts the existence of $\mathfrak Y$ as required.
\end{proof}

\begin{proposition}\label{irrpres} Let $\mathcal O$ be a $G$-orbit of $\mathfrak N$ and let $\mathbb O$ be a $G$-orbit of $\mathcal N$.
\item {\bf 1)} Let $\mathfrak X \in \mathrm{Comp} ( \mathcal O )$. The variety $\mathfrak X _{\Bbbk}$ is irreducible.
\item {\bf 2)} Let $\mathfrak Y \in \mathrm{Comp} ( \mathbb O )$ be such that $\mathfrak Y _{\mathbb K} \in \mathrm{Comp} ( \mathbb O _{\mathbb K} )$. Then, $\mathfrak Y _{\Bbbk}$ is irreducible.
\end{proposition}

\begin{proof}
Let $X \in \mathcal O ( A )$. For two irreducible components $\mathcal G _X ^i, \mathcal G _X ^j \subset \mathcal G _X$, we set
$$Z _0 := G ( \mathcal G _X ^i / B \times \mathcal G _X ^j / B ) \subset Z.$$
We have
\begin{equation}
\dim Z _0 = \dim \mathcal G _X ^i / B + \dim \mathcal G _X ^j / B + \dim \mathcal O = \dim Z.\label{Stest}
\end{equation}
Hence, $\overline{Z _0} \subset Z$ is an irreducible component of $Z$. It follows that $( \mathcal G _X ^i ) _{\Bbbk} \cap ( \mathcal G _X ^j ) _{\Bbbk}$ does not contain an irreducible component of $( \mathcal G _X ) _{\Bbbk}$ unless $\mathcal G _X ^i = \mathcal G _X ^j$. By taking quotient by $\mathsf{Stab} _G ( X )$, we conclude the first assertion.

For the second assertion, we uniformly change $\mathfrak N$ by $\mathcal N$ and $\mathbb V ^+$ by $\mathfrak n$. Then, we use the extra assumption to guarantee the flatness of $\mathcal G _X ^i$. Set $C _X$ to be the component group of $\mathsf{Stab} _{G _{\mathbb K}} X$. Then, the problem reduces to show:
\begin{itemize}
\item[$(\clubsuit)$] $\mathfrak Y _{\Bbbk} = ( \mathsf{Stab} _{G _{\Bbbk}} X _{\Bbbk} ) \backslash ( \mathcal G _X ^+ ) _{\Bbbk}$ is irreducible for each $C _X$-orbit $\mathcal G _X ^{+}$ of $\mathcal G _X ^i$.
\end{itemize}
(Notice that $C _X$ is not defined over $A$, and hence $\mathcal G _X ^+$ cannot defined directly.) By Lemma \ref{ordorb}, an irreducible component $\mathfrak Y _{\Bbbk} ^{\prime} \subset \mathfrak Y _{\Bbbk}$ is a union of subvarieties of orbital varieties of some $G _{\Bbbk}$-orbits of $\mathbb O_{\Bbbk}$. We have $\dim \mathfrak Y _{\Bbbk} ^{\prime} \le \dim \mathfrak Y _{\mathbb K}$ with the equality holds only when $\mathfrak Y _{\Bbbk} ^{\prime} \in \mathrm{Comp} ( \mathbb O _{\Bbbk} ^{\prime} )$. By the upper-semicontinuity of fiber dimensions applied to $\mathfrak Y$ over $A$, we have necessarily $\dim \mathfrak Y _{\Bbbk} ^{\prime} = \dim \mathfrak Y _{\mathbb K}$. In particular, we have an irreducible component $\mathcal G _X ^0 \subset \mathcal G _X ^{+}$ corresponding to $\mathrm{Comp} ( \mathbb O _{\Bbbk} ^{\prime} )$. Assume that $X _{\Bbbk} \in \mathbb O _{\Bbbk} ^{\prime} ( \Bbbk )$. Then, $G_{\Bbbk} ( \mathcal G _X ^0 \times \mathcal G _X ^0 ) _{\Bbbk}$ is a union of irreducible components of $Z _{\Bbbk}$ thanks to the dimension estimate as in (\ref{Stest}). Therefore, we deduce that $\overline{G_{\Bbbk} ( \mathcal G _X ^0 \times \mathcal G _X ^0 ) _{\Bbbk}} \subset ( \overline{G _{\mathbb K} ( \mathcal G _X ^0 \times \mathcal G _X ^0 )_{\mathbb K}} ) _{\Bbbk} \subset Z _{\Bbbk}$ is an irreducible component, which implies that $( \mathcal G _X ^0 ) _{\Bbbk}$ is irreducible. It follows that an irreducible component of $\mathrm{Comp} ( \mathbb O _{\Bbbk} ^{\prime} )$ obtained as
$$( \mathsf{Stab} _{G _{\Bbbk}} X _{\Bbbk} ) \backslash ( \mathcal G _X ^0 )_{\Bbbk} = ( \overline{( \mathsf{Stab} _{G _{\mathbb K}} X _{\mathbb K} )^{\circ} \backslash ( \mathcal G _X ^0 ) _{\mathbb K}} )_{\Bbbk} = ( \overline{\mathsf{Stab} _{G _{\mathbb K}} X _{\mathbb K} \backslash ( \mathcal G _X^+ ) _{\mathbb K}} )_{\Bbbk}$$
is unique as required.
\end{proof}

\section{Exotic orbital varieties: proof}\label{eovp}
This section is devoted to the proof of Theorem \ref{equigen}.

The proof itself is a modification of the arguments of Steinberg \cite{Ste}, Spaltenstein \cite{S77}, and Joseph \cite{J}. The only essential diffusion in the proof is contained in the strict semi-smallness of the map $\nu$, which follows from \cite{K1} \S 1 and \S 8. Since the literature is little scattered, we provide a proof with its necessary modifications.

In the below, we assume the same settings as in Theorem \ref{equigen}, but we drop the subscript ${} _{\mathbb K}$ or ${} _{\Bbbk}$ for the sake of simplicity.

\begin{lemma}\label{cgineq}
We have $\dim \mathfrak X \le \frac{1}{2} \dim \mathcal O$. Moreover, there exists $\mathfrak X \in \mathrm{Comp} ( \mathcal O )$ which satisfies the equality.
\end{lemma}
\begin{proof}
Let $X \in \mathfrak X$. We have
$$\frac{1}{2} \dim G \mathfrak X + \dim \nu ^{-1} ( X ) = \dim G / B$$
by the (strict) semi-smallness of $\nu$. Since $\nu ^{-1} ( X ) = \mathcal G _X / B$, we have
$$\frac{1}{2} \dim G \mathfrak X + \dim \mathcal G _X - \dim B = \dim G / B.$$
We have $\mathfrak X \subset \mathcal G _X / \mathrm{Stab} _{G} ( X )$. In particular, we have
$$\frac{1}{2} \dim G \mathfrak X + \dim \mathfrak X + \dim \mathrm{Stab} _{G} ( X ) \le \dim G.$$
Therefore, we have
\begin{eqnarray}
\dim \mathfrak X \le \dim G - \dim \mathrm{Stab} _{G} ( X ) - \frac{1}{2} \dim G \mathfrak X = \frac{1}{2} \dim \mathcal O,\label{dimestO}
\end{eqnarray}
which proves the first assertion. The second assertion follows by choosing $\mathfrak X$ so that $\dim \mathfrak X = \dim \mathcal G _X / \mathrm{Stab} _{G} ( X )$.
\end{proof}

\begin{proposition}\label{descwseed}
Assume that $\dim \mathfrak X = \frac{1}{2} \dim \mathcal O$. Then there exists $w \in W$ such that
$$\mathfrak X \subset \overline{B ( \mathbb V ^+ \cap {} ^w \mathbb V ^+ )}$$
is a dense open subset.
\end{proposition}

\begin{proof}
Let $X \in \mathfrak X$. We assume that $\mathfrak X \cong \mathcal G ^i X$ for an irreducible component $\mathcal G ^i$ of $\mathcal G _X \subset G$. We put $\mathcal E _X ^i := \mathcal G ^i / B$, which is an irreducible component of $\mathcal E _X$. By $\dim \mathfrak X = \frac{1}{2} \dim \mathcal O$, it follows that $\mathcal E _X ^i$ has the maximal dimension among the irreducible components of $\mathcal E _X$. In other words, we have
\begin{eqnarray}
\dim \mathcal E _X ^i = \frac{1}{2} ( \dim \mathfrak N - \dim \mathcal O ). \label{SSS}
\end{eqnarray}
Consider the variety
$$\mathcal S := \{ ( g _1 B, g _2 B, v ) \in \mathcal B \times \mathcal B \times \mathbb V; v \in g _1 \mathbb V ^+ \cap g _2 \mathbb V ^+ \cap \mathcal O \}$$
and its subvarieties
$$\mathcal S _w := \{ ( g _1 B, g _2 B, v ) \in \mathcal S; g _1 ^{-1} g _2 \in B \dot{w} B \}$$
for each $w \in W$. It is straight-forward to check $\mathcal S = \sqcup _{w \in W} \mathcal S _w$ (the arguments in \cite{Ste} p133 L14-L20 works merely by changing the meaning of the symbols appropriately). By considering the third projection $p _3 : \mathcal S \rightarrow \mathcal O$, we deduce that
$$p _3 ^{-1} ( X ) \cong \mathcal E _X \times \mathcal E _X.$$
Consider the projection $p _{12} : \mathcal S \rightarrow \mathcal B \times \mathcal B$ of $\mathcal S$ to the first two components. By definition, we have $p _{12} ( \mathcal S _w ) = G ( [B \times \dot{w} B ] )$ or $\emptyset$. It follows that
\begin{align*}
\dim \mathcal S _w = & \dim \mathcal B + \ell ( w ) + \dim ( \mathbb V ^+ \cap {} ^w \mathbb V ^+ \cap \mathcal O )\\
& \le \dim \mathcal B + \ell ( w ) + \dim ( \mathbb V ^+ \cap {} ^w \mathbb V ^+ ) = \dim G - n = \dim \mathcal S
\end{align*}
whenever $\mathcal S _w \neq \emptyset$. Define
$$\mathcal S ^{i,i} := G ( \mathcal E _X ^i \times \mathcal E _X ^i \times \{ X \} ) \subset G p _3 ^{-1} ( X ) \subset \mathcal S.$$
This is an irreducible component of $\mathcal S$. Since $G p _3 ^{-1} ( X ) = \mathcal S$ and $(\ref{SSS})$, we conclude
$$\dim \mathcal S ^{i, i} = \dim \mathcal O + 2 \dim \mathcal E _X ^i = \dim \mathcal O + 2 \times \frac{1}{2} \mathrm{codim} \mathcal O = \mathfrak N = \dim G - n.$$
There exists $w \in W$ such that $\mathcal S _w \cap \mathcal S ^{i, i} \subset \mathcal S ^{i, i}$ is a dense open subset. By dimension counting, we deduce $\mathcal S ^{i, i} = \overline{\mathcal S _w}$. Now we have
\begin{eqnarray}
\mathcal S _w = \{ ( g _1 B, g _2 B, v ) ; g _1 ^{-1} g _2 \in B \dot{w} B, v \in g _1 ( \mathbb V ^+ \cap {} ^w \mathbb V ^+ \cap \mathcal O )\}. \label{indep}
\end{eqnarray}
Since $\dim \mathcal S = \dim \mathcal S ^{i, i} = \dim \mathcal S _w$, we deduce that
$$\mathbb V ^+ \cap {} ^w \mathbb V ^+ \cap \mathcal O \subset \mathbb V ^+ \cap {} ^w \mathbb V ^+$$
is dense.\\
Consider the image $\mathcal G _w$ of $\mathcal S _w$ under the first and third projection $p_{13} : \mathcal S \rightarrow \mathcal B \times \mathbb V$. Its second projection $q _3 : \mathcal G _w \rightarrow \mathcal O$ satisfies $q _3 \circ p _{13} = p _3$. In the RHS of (\ref{indep}), $g _2$ plays no r\^ole for the restriction on $v$. Therefore, we deduce $q _3 ^{-1} ( X ) \subset \mathcal E _X ^i$ (dense open subset). By construction, we have
$$q _3 ^{-1} ( X ) = \{ ( g B, X ) ; X \in g ( \mathbb V ^+ \cap {} ^w \mathbb V ^+ \cap \mathcal O )\}.$$
As a consequence, we deduce
$$\mathfrak X = \mathcal G ^i X \subset \overline{\{ g ^{-1} X ; ( g B, X ) \in q _3 ^{-1} ( X ) \}} = \overline{B ( \mathbb V ^+ \cap {} ^w \mathbb V ^+ )}.$$
Since the second inclusion is dense by construction, we conclude the result.
\end{proof}

Lemma \ref{cgineq} and Proposition \ref{descwseed} claim that the both assertions of Theorem \ref{equigen} hold for at least one $\mathfrak X \in \mathrm{Comp} ( \mathcal O )$. To derive Theorem \ref{equigen} for general irreducible components, we need some preparation:

For each $1 \le i \le n$, we put $\mathbb V ( i ) := U ( \mathfrak b ) {} ^{s _i} \mathbb V ^+$. We define $\mathbb V ^+ _i := \mathbb V ^+ / ( \mathbb V ^+ \cap {} ^{s _i} \mathbb V ^+ )$ and $\mathbb V _i := \mathbb V ( i ) / ( \mathbb V ^+ \cap {} ^{s _i} \mathbb V ^+ )$.

For each $1 \le i \le n$, we put $P _{i} := B \dot{s} _i B \cup B$. It is a parabolic subgroup of $G$.

The derived group of the Levi part of $P _i$ is isomorphic to $\mathop{SL} ( 2 )$. Its action on $\mathbb V _i$ is equivalent to either $\mathfrak{sl} ( 2 )$ (adjoint representation, $1 \le i < n$) or $\mathbb K ^2$ (vector representation, $i = n$).

Since the both of $\mathbb V ( i )$ and $( \mathbb V ^+ \cap {} ^{s _i} \mathbb V ^+ )$ are $P _i$-stable, it follows that $\mathbb V _i$ admits a natural $P _i$-action. Let $\pi _{i} : \mathbb V ( i ) \longrightarrow \mathbb V _i$. The map $\pi _i$ is $P _i$-equivariant. We define $\mathfrak X _i := \pi _i ( \mathfrak X )$.

\begin{lemma}\label{inheritance}
Let $1 \le i \le n$. Assume that the both assertions of Theorem \ref{equigen} hold for $\mathfrak X \in \mathrm{Comp} ( \mathcal O )$. Then, $P _i \mathfrak X \cap \mathbb V ^+$ is a union of elements of $\mathrm{Comp} ( \mathcal O )$ which satisfy the both assertions of Theorem \ref{equigen}.
\end{lemma}

\begin{proof}
By Proposition \ref{descwseed}, it suffices to verify Theorem \ref{equigen} 1).

By construction, $\mathfrak X _i \subset \mathbb V ^+ _i$ is a $B$-stable subset. We have $\dim \mathbb V ^+ _i = 1$. Hence, we have $\mathfrak X _i = \{ 0 \}$ or $\overline{\mathfrak X _i} = \mathbb V ^+ _i$. If $\mathfrak X _i = \{ 0 \}$, then $\mathfrak X$ is $P _i$-stable. Thus, the assertion trivially holds.

Therefore, we concentrate ourselves to the case $\overline{\mathfrak X _i} = \mathbb V ^+ _i$ in the below. Let $\mathfrak X ^{\prime} \in \mathrm{Comp} ( \mathcal O )$ such that $\mathfrak X ^{\prime} \cap P _i \mathfrak X \not\subset \mathfrak X$. If $\mathfrak X ^{\prime}$ does not exist, then we have $P _i \mathfrak X \cap \mathbb V ^+ = \mathfrak X$. Hence, the assertion trivially holds. Thus, we assume the existence of $\mathfrak X ^{\prime}$.

Let $\mathfrak D := \mathfrak X \cap \pi _i ^{-1} ( \{ 0 \} ) \subset \mathfrak X$. This is a purely codimension one subscheme of $\mathfrak X$. Since $\pi _i ^{-1} ( \{ 0 \} ) = ( \mathbb V ^+ \cap {} ^{s _i} \mathbb V ^+ )$ is $P _i$-stable, it follows that
$$P _i \mathfrak D \subset P _i \mathfrak X \cap \pi _i ^{-1} ( \{ 0 \} ) \subset \mathbb V ^+.$$

Let $0 \neq X \in \mathbb V _i ^+$. By an explicit $\mathop{SL} ( 2 )$-computation, we have $g X \in \mathbb V _i ^+$ ($g \in P _i$) if and only if $g \in P _i \cap B$. This implies
$$P _i ( \mathfrak X - \mathfrak D ) \cap \mathbb V ^+ = \mathfrak X - \mathfrak D.$$
Hence, we have $\mathfrak X ^{\prime} \cap P _i \mathfrak D \neq \emptyset$.  Let $\mathfrak D _0 \subset \mathfrak D$ be an irreducible component such that $\mathfrak X ^{\prime} \cap P _i \mathfrak D _0 \not\subset \mathfrak D$. We have necessarily $P _{i} \mathfrak D _0 \neq \mathfrak D _0$. This implies
$$\dim P _{i} \mathfrak D _0 = \dim \mathfrak D _0 + 1 = \dim \mathfrak X.$$
As a consequence, $\overline{P _i \mathfrak D _0}$ contains a (unique) element of $\mathrm{Comp} ( \mathcal O )$ which is different from $\mathfrak X$. Letting $\mathfrak X ^{\prime}$ and $\mathfrak D _0$ vary arbitrary, we conclude the result.
\end{proof}

In order to complete the proof of Theorem \ref{equigen}, it suffices to prove that a successive application of Lemma \ref{inheritance} eventually exhausts the whole of $\mathrm{Comp} ( \mathcal O )$. This is guaranteed by the following:

\begin{proposition}\label{linkage}
Let $\mathfrak X, \mathfrak X ^{\prime} \in \mathrm{Comp} ( \mathcal O )$. Assume that $\mathfrak X ^{\prime}$ satisfies the both assertions of Theorem \ref{equigen}. Then, there exists a sequence of integers $i _1, i _2, \ldots, i _{m} \in [ 1, n ]$ and a sequence $\mathfrak X _1, \mathfrak X _2, \ldots, \mathfrak X _m \in \mathrm{Comp} ( \mathcal O )$ such that
$$\mathfrak X ^{\prime} = \mathfrak X _1, \hskip 2mm \mathfrak X = \mathfrak X _m, \text{ and } \mathfrak X _{k - 1} \subset P _{i _k} \mathfrak X _{k}$$
hold for every $2 \le k \le m$.
\end{proposition}

\begin{proof}
Let $i _1, \ldots i _m \in [ 1, n ]$ be a sequence of integers such that
\begin{eqnarray}
\mathfrak X ^{\prime} \subset P _{i _1} P _{i _2} \cdots P _{i _m} \overline{\mathfrak X}. \label{incl-seq}
\end{eqnarray}
We assume that $(\star)$: $\mathfrak X ^{\prime} \not\subset P _{i _1 ^{\prime}} P _{i _2 ^{\prime}} \cdots P _{i _{m ^{\prime}} ^{\prime}} \mathfrak X$ does not holds for any sequence $i _1 ^{\prime}, \ldots, i _{m ^{\prime}} ^{\prime}$ if $m ^{\prime} < m$. This implies that $s _{i _1} s _{i _2} \cdots s _{i _m}$ is a reduced expression. We prove that there exists a sequence $\mathfrak X _1, \mathfrak X _2, \ldots, \mathfrak X _m \in \mathrm{Comp} ( \mathcal O )$ which satisfies the required condition. By $(\ref{incl-seq})$ and $(\star)$, we deduce
$$\mathfrak Z := B \dot{s} _{i _1} \mathfrak X ^{\prime} \cap B \dot{s} _{i _2} B \cdots B \dot{s} _{i _m} \overline{\mathfrak X} \neq \emptyset.$$
\begin{claim}
We have $\dim \mathfrak Z = \dim \mathfrak X ^{\prime}$.
\end{claim}

\begin{proof}
By $(\star)$, we have $P _{i _1} \mathfrak X ^{\prime} \neq \mathfrak X ^{\prime}$. Hence, we have
$$P _{i _1} \mathfrak X ^{\prime} = \dim \mathfrak X ^{\prime} +1.$$
Since $\mathfrak Z$ is an open subset of a codimension one subscheme of $\overline{P _{i _1} \mathfrak X ^{\prime}}$, we deduce the result. 
\end{proof}

We return to the proof of Proposition \ref{linkage}.

We have $B \dot{s} _{i _1} \mathfrak X ^{\prime} \subset \mathbb V ( i _1 )$. We put $w = s _{i _2} \cdots s _{i _m}$. This is a reduced expression. Since $\ell ( w ) < \ell ( s _{i _1} w )$, we have $\mathbf v [ - \alpha _{i _1} ^{\vee} ] \not\in \dot{w} \mathbb V ^+$, where $\alpha _i ^{\vee} = \alpha _i$ $(i \neq n)$ or $\epsilon _n$ $(i=n)$. It follows that
$$( \mathbb K ^{\times} \mathbf v [ - \alpha _{i _1} ^{\vee} ] + \mathbb V [ 0 ] + \mathbb V ^+ ) \cap B \dot{w} \mathbb V ^+ = \emptyset$$
by a weight comparison. (Here we need to replace $\mathbb K$ with $\Bbbk$ when the subscripts are $\Bbbk$.) Hence, we have
$\mathbb V ( i _1 ) \cap B \dot{w} \mathbb V ^+ \subset \mathbb V ^+$. Taking account into $(\star)$, this implies that
$$\emptyset \neq B \dot{s} _{i _1} \mathfrak X ^{\prime} \cap B \dot{w} \overline{\mathfrak X} \subset \mathbb V ^+.$$
In particular, there exists an irreducible component $\mathfrak X ^{\prime \prime} \subset P _{i _1} \mathfrak X ^{\prime} \cap \mathbb V ^+$ such that
$$\mathfrak X ^{\prime \prime} \cap B \dot{s} _{i _2} B \cdots B \dot{s} _{i _m} \overline{\mathfrak X} \neq \emptyset, \mathfrak X ^{\prime} \subset P _{i _1} \mathfrak X ^{\prime\prime},$$
and the intersection at the most LHS is a maximal dimensional irreducible component of $\mathfrak Z$. By Lemma \ref{inheritance}, it suffices to prove the assertion for $\mathfrak X, \mathfrak X ^{\prime\prime} \in \mathrm{Comp} ( \mathcal O )$ with
$$\mathfrak X ^{\prime \prime} \subset \overline{B \dot{s} _{i _2} B \cdots B \dot{s} _{i _m} \mathfrak X} = P _{i _2} P _{i _3} \cdots P _{i _m} \overline{\mathfrak X}.$$
Since the assertion for $m = 1$ is proved in Lemma \ref{inheritance}, the downward induction on $m$ yields the result.
\end{proof}

\section{Comparison of Springer correspondences}

We work under the same setting as in \S \ref{NT}.

Let $\mathcal X$ be a $T$-equivariant scheme over $A$. Let $K ^{T _{*}} ( \mathcal X _{*} ) _{\mathbb Q}$ be the $\mathbb Q$-coefficient Grothendieck group of $T _{*}$-equivariant coherent sheaves on $\mathcal X _{*}$ which are flat over the base ($* = A, \mathbb K, \Bbbk$). Let $R ( T ) _{\mathbb Q}$ be the representation ring of $T$ with coefficient $\mathbb Q$. For a $T$-module $V$, we define $\mathrm{ch} _T V$ to be the class $[ V ] \in R ( T ) _{\mathbb Q}$.
Consider a map
$$p : K ^T ( \mathfrak n ) _{\mathbb Q} \longrightarrow R ( T ) _{\mathbb Q}$$
which sends a $T$-equivariant closed subset $C \subset \mathfrak n$ to the ratio
$$\mathrm{ch} _T \Gamma ( \mathfrak n, \mathcal O _{C} ) / \mathrm{ch} _T \Gamma ( \mathfrak n, \mathcal O _{\mathfrak n} ) \in R ( T ) _{\mathbb Q}.$$
Replacing $T$ and $\mathfrak n$ with $T _{*}$ and $\mathfrak n _*$ (where $* = A, \mathbb K, \Bbbk$), we define the corresponding maps $p _{*}$.
Similarly, consider a map
$$q : K ^T ( \mathbb V ^+ ) _{\mathbb Q} \longrightarrow R ( T ) _{\mathbb Q}$$
which sends a $T$-equivariant closed subset $C \subset \mathbb V ^+$ to the ratio
$$\mathrm{ch} _T \Gamma ( \mathbb V ^+, \mathcal O _{C} ) / \mathrm{ch} _T \Gamma ( \mathbb V ^+, \mathcal O _{\mathbb V ^+} ) \in R ( T ) _{\mathbb Q}.$$ Replacing $T$ and $\mathbb V ^+$ with $T _{*}$ and $\mathbb V _* ^+$ (where $* = \mathbb K, \Bbbk$), we define the corresponding maps $q _{*}$.

Let $\mathsf{fx} : R ( T ) \rightarrow \mathbb C [[ \mathfrak t ]]$ be the map given by the formal expansion of a function on $T$ along $1$. For $f \in R ( T )$, we denote the lowest non-zero homogeneous term of $\mathsf{fx} ( f )$ by $\mathsf{lt} ( f )$. By definition, $\mathsf{lt} ( f )$ is a homogeneous polynomial on $\mathfrak t$.

Let $\mathbb O$ be a $G$-orbit in $\mathcal N$. For each $\mathfrak Y _* \in \mathrm{Comp} ( \mathbb O _* )$ ($* = A, \mathbb K, \Bbbk$), we define the Joseph polynomial attached to $\mathfrak Y _*$ as $\mathsf{lt} \left( p _* ( \overline{\mathfrak Y} _* ) \right)$. Let $\mathcal O$ be a $G$-orbit in $\mathfrak N$. For each $\mathfrak X _* \in \mathrm{Comp} ( \mathcal O _* )$ ($* = \mathbb K, \Bbbk$), we define the Joseph polynomial attached to $\mathfrak X _*$ as $\mathsf{J} ( \mathfrak X _{*} ) := \mathsf{lt} \left( q _* ( \overline{\mathfrak X} _* ) \right)$.

We denote the set of $\mathbb Q$-multiples of Joseph polynomials attached to orbital varieties of $\mathcal O _*$ or $\mathbb O _*$ ($* = A, \mathbb K, \Bbbk$) by $\mathrm{Jos} ( \mathcal O _* )$ or $\mathrm{Jos} ( \mathbb O _* )$, respectively.

\begin{proposition}\label{flatK}
Let $\mathcal X \subset \mathbb V ^+$ and $\mathcal Y \subset \mathfrak n$ be $T$-equivariant flat subfamilies over $A$. Then, we have
$$q _{\mathbb K} ( \mathcal X _{\mathbb K} ) = q _{\Bbbk} ( \mathcal X _{\Bbbk} ) \text{ and } p _{\mathbb K} ( \mathcal Y _{\mathbb K} ) = p _{\Bbbk} ( \mathcal Y _{\Bbbk} ).$$
\end{proposition}

\begin{proof}
Each character of tori is defined over $A$. In particular, specialization (to $\mathbb K$ or $\Bbbk$) of a $T$-equivariant flat $A$-module of rank one preserves the character. Hence, the assumption implies that the coordinate rings $\mathbb K [ \mathcal X _{\mathbb K} ]$ and $\Bbbk [ \mathcal X _{\Bbbk} ]$ share the same character. Hence, we conclude the result for $\mathcal X$. The case $\mathcal Y$ is entirely the same.
\end{proof}

\begin{proposition}\label{deformK}
Let $\mathcal C$ be a $T _{\Bbbk}$-stable flat subfamily of $\mathcal V ^+$ over $\mathbb A ^1 _{\Bbbk}$ whose fibers are irreducible schemes. Let $\mathcal C _t := C \dot{\cap} \pi ^{-1} ( t )$. Then, we have
$$\mathsf{lt} ( p _{\Bbbk} ( [ \mathcal C _1 ] ) ) \in \mathbb Z _{\ge 1} \mathsf{lt} ( q _{\Bbbk} ( [ \mathsf{F} _1 ^{-1} ( \mathcal C _0 ) ] ) ).$$
\end{proposition}

\begin{proof}
For the sake of simplicity, we drop the subscripts ${} _{\Bbbk}$ during this proof. A $T$-character does not admit a non-trivial deformation. We express $\mathcal O _{\mathcal C _0}$ by a $T$-equivariant free resolution
$$0 \to \mathcal F _{N} \to \cdots \to \mathcal F _2 \to \mathcal F _1 \to \mathcal O _{\mathcal C _0} \to 0$$ 
such that $N < \infty$ and each $\mathcal F _i$ is isomorphic to a product of $T$-module and $\mathcal O _{\mathsf F _1 ( \mathbb V ^+ )}$ (c.f. Chriss-Ginzburg \cite{CG} \S 5.1 or Sumihiro \cite{Su}). Utilizing the Frobenius splitting of $V _1$, we deduce that $\mathcal O _{\mathbb V ^+}$ is a free $\mathcal O _{\mathsf{F} _1 ( \mathbb V ^+ )}$-module of rank $2 ^n$. It follows that $q ( [ \mathsf{F} _1 ^* ( \mathcal C _0 ) ] ) = p ( [ \mathcal C _1 ] )$. Now the sheaf $\mathsf F _1 ^* ( \mathcal O _{\mathcal C _0} )$ is a vector bundle of rank $m$ along an open subset of $\mathsf F _1 ^{-1} ( \mathcal C _0 )$ for some $m > 0$. Taking the lowest term of $q$ neglects the effects from subsets of $\mathsf F _1 ^{-1} ( \mathcal C _0 )$ which has codimension $\ge 1$. Therefore, we conclude that
$$m \mathsf{lt} q ( [ \mathsf{F} _1 ^{-1} ( \mathcal C _0 ) ] ) = \mathsf{lt} q ( [ \mathsf{F} _1 ^* ( \mathcal C _0 ) ] ) = \mathsf{lt} p ( [ \mathcal C _1 ] )$$
as desired.
\end{proof}

\begin{theorem}[Joseph \cite{J, J2}]\label{JosK}
Let $\mathbb O$ be a $G$-orbit of $\mathcal N$ and let $\mathcal O$ be a $G$-orbit of $\mathfrak N$. The $\mathbb C$-span of $\mathrm{Jos} ( \mathbb O _{\mathbb K} )$ or $\mathrm{Jos} ( \mathcal O _\mathbb K )$ form an irreducible $W$-module with a basis $\mathrm{Jos} ( \mathbb O _{\mathbb K} )$ or $\mathrm{Jos} ( \mathcal O _\mathbb K )$.
\end{theorem}

\begin{proof}
The proof for the case $\mathbb O _{\mathbb K}$ is the original case and is treated in \cite{CG} 6.5.13 and 7.4.1. The case $\mathcal O _{\mathbb K}$ follows from the same construction as in \cite{CG} 6.5 and 7.4 if we replace $\mathfrak n _{\mathbb K}$ with $\mathbb V _{\mathbb K} ^+$, $\mathcal N _{\mathbb K}$ with $\mathfrak N _{\mathbb K}$, and $\mathbb O _{\mathbb K}$ with $\mathcal O _{\mathbb K}$ uniformly.
\end{proof}

\begin{corollary}[$(\spadesuit)_1$ in \S \ref{eovs}]\label{PW}
We have $\sum _{\mathcal O _{\mathbb K} \in G _{\mathbb K} \backslash \mathfrak N _{\mathbb K}} \left( \mathrm{Comp} ( \mathcal O _{\mathbb K} ) \right) ^2 = \# W$.
\end{corollary}

\begin{proof}
Since our exotic Springer correspondence is a bijection between the orbits of $\mathfrak N$ and $W ^{\vee}$, Theorem \ref{JosK} and the Wedderburn theorem yields the result.
\end{proof}

\begin{corollary}[$(\spadesuit)_2$ in \S \ref{eovs}]\label{Jind}
For a $G$-orbit $\mathcal O$ of $\mathfrak N$, there are at least $\# \mathrm{Comp} ( \mathcal O _{\mathbb K})$ orbital varieties of $\mathcal O _{\Bbbk}$ contained in the closure of that of $\mathcal O _{\mathbb K}$.
\end{corollary}

\begin{proof}
Straight-forward consequence of Proposition \ref{flatK} and Theorem \ref{JosK}. 
\end{proof}

\begin{theorem}\label{Josmain}
Let $\mathbb O$ be a $G$-orbit in $\mathcal N$. Let $\mathcal O$ be a $G$-orbit in $\mathfrak N$ such that $\mathsf{df} ( \mathcal O  _{\Bbbk} ) \subset \mathbb O _{\Bbbk}$ is a open dense subset. Then, we have
$$\mathrm{Jos} ( \mathbb O _{\mathbb K} ) = \mathrm{Jos} ( \mathcal O _{\mathbb K} ).$$
\end{theorem}

\begin{proof}
Since the construction of Joseph polynomials factors through the closures of orbital varieties, we may refer an orbital variety closure as an orbital variety during this proof (for the sake of simplicity). We prove the following identities:
\begin{eqnarray}
\mathrm{Jos} ( \mathbb O _{\mathbb K} ) = \mathrm{Jos} ( \mathbb O _{\Bbbk} ) = \mathrm{Jos} ( \mathcal O _{\Bbbk} ) = \mathrm{Jos} ( \mathcal O _{\mathbb K} ).\label{tight-rope}
\end{eqnarray}
\item (Proof of $\mathrm{Jos} ( \mathbb O _{\mathbb K} ) \subset \mathrm{Jos} ( \mathbb O _{\Bbbk} )$) Let $\mathfrak Y \in \mathrm{Comp} ( \mathbb O )$ be such that $\mathfrak Y _{\mathbb K} \in \mathrm{Jos} ( \mathbb O _{\mathbb K} )$. The variety $\mathfrak Y _{\Bbbk}$ is irreducible by Proposition \ref{irrpres} 2). Since $\mathfrak Y$ dominates $A$, we deduce that $\mathfrak Y$ is a flat over $A$. Therefore, $\mathrm{Jos} ( \mathbb O _{\mathbb K} ) \subset \mathrm{Jos} ( \mathbb O _{\Bbbk} )$ follows from Proposition \ref{flatK} as desired.

\item (Proof of $\mathrm{Jos} ( \mathbb O _{\Bbbk} ) = \mathrm{Jos} ( \mathcal O _{\Bbbk} )$) Let $\mathfrak X _{\Bbbk} \in \mathrm{Comp} ( \mathcal O _{\Bbbk} )$. Consider a family $\mathsf{ml} ( \overline{\mathfrak X} _{\Bbbk} )$. This is a $B _{\Bbbk}$-stable equidimensional subfamily of $\mathcal V ^+$. By the comparison of dimensions, it is a flat family of orbital varieties over $\mathbb A ^1 _{\Bbbk}$. Hence the equality $\mathrm{Jos} ( \mathbb O _{\Bbbk} ) = \mathrm{Jos} ( \mathcal O _{\Bbbk} )$ follows from Proposition \ref{deformK}.

\item (Proof of $\mathrm{Jos} ( \mathcal O _{\mathbb K} ) = \mathrm{Jos} ( \mathcal O _{\Bbbk} )$) Let $\mathfrak X \in \mathrm{Comp} ( \mathcal O )$.
The variety $\mathfrak X _{\Bbbk}$ is irreducible by Proposition \ref{irrpres} 1). Since $\mathfrak X$ is flat over $A$, the equality $\mathrm{Jos} ( \mathcal O _{\mathbb K} ) = \mathrm{Jos} ( \mathcal O _{\Bbbk} )$ follows from Proposition \ref{flatK}.\\
Thanks to Theorem \ref{JosK}, we deduce (\ref{tight-rope}), which implies the result.
\end{proof}

Let $\mathbb O _{\mathbb K}$ be a $G _{\mathbb K}$-orbit of $\mathcal N _{\mathbb K}$. Let $Y \in \mathbb O _{\mathbb K}$. We define
$$\mathcal B _Y := \{ g \in G _{\mathbb K} ; Y \in \mathrm{Ad} ( g ) \mathfrak n _{\mathbb K} \} / B _{\mathbb K} \subset \mathcal B _{\mathbb K}$$
and call it the Springer fiber along $Y$.

\begin{corollary}
Let $\mathbb O$ be a $G$-orbit of $\mathcal N$ with its codimension $2 d$. Let $\mathcal O$ be a $G$-orbit of $\mathfrak N$ such that $\mathsf{df}( \mathcal O _{\Bbbk} ) \subset \mathbb O _{\Bbbk}$ is a dense open subset. Let $X \in \mathcal O _{\mathbb K}$ and let $Y \in \mathbb O _{\mathbb K}$. Let $C _Y := \mathrm{Stab} _{G _{\mathbb K}} ( Y ) / \mathrm{Stab} _{G _{\mathbb K}} ( Y ) ^{\circ}$. Then, we have a $W$-equivariant isomorphism
$$H _{2 d} ( \mathcal B _{Y}, \mathbb C ) ^{C _Y} \cong H _{2 d} ( \mathcal E _X, \mathbb C ),$$
compatible with their embeddings into $H _{2d} ( \mathcal B, \mathbb C )$. Moreover, the bases given by irreducible components of $\mathcal B _{Y}$ and $\mathcal E _X$ coincide up to scalar multiplication.
\end{corollary}

\begin{proof}
This is a direct consequence of Theorem \ref{Josmain} and \cite{CG} 6.5.13. Here the counter-part of \cite{CG} 6.5.13 for $\mathfrak N$ is obtained by merely by replacing the meaning of the symbols as $\mathcal N$ by $\mathfrak N$, $\mathbb O$ by $\mathcal O$, and $\mathcal B _Y$ by $\mathcal E _X$.
\end{proof}

\section{An explicit description of the correspondence}

Keep the setting of the previous section, but fix the base to be $\mathbb K$ (or rather its scalar extension to $\mathbb C$). Here we use the notation $\mathbb V ^{\bm \lambda}, \mathbb V ^{\bm \lambda} _{01}, \ldots$ defined in \S \ref{seW}.

We start from a dimension formula which seems to go back to Kraft-Procesi \cite{KP3} \S 8.1. Here we present a slightly modified form which is suitable for applications.

\begin{theorem}[Kraft-Procesi \cite{KP3}]\label{dim_formula}
Let $\lambda$ be a partition of $n$ and let $\mathbf 0 = ( 0, 0, \ldots )$ be a sequence of zeros. We put ${\bm \lambda} := ( \lambda, \mathbf 0 )$. Then, we have
$$\dim \mathcal O _{( \lambda, \mathbf 0 )} = 2 \dim ( \mathcal O _{( \lambda, \mathbf 0 )} \cap \mathbb V _0 ) = 4 \sum _{i < j} ( d _i ^{{\bm \lambda}} - d _{i-1} ^{{\bm \lambda}} ) ( d _j ^{{\bm \lambda}} - d _{j-1} ^{{\bm \lambda}} ),$$
where $\{ d _i ^{{\bm \lambda}} \} _{i \ge 1}$ is a sequence obtained from $\bm \lambda$ as in Definition \ref{se}.
\end{theorem}

\begin{lemma}\label{irrcomp}
For each marked partition ${\bm \lambda}$, there exists $\mathfrak X _{{\bm \lambda}} \in \mathrm{Comp} ( \mathcal O _{{\bm \lambda}} )$ such that
$$\overline{\mathfrak X _{{\bm \lambda}}} = \overline{B \mathbb V ^{{\bm \lambda}}}.$$
\end{lemma}

\begin{proof}
We set $d ^{\nabla} _i := d _i ^{\bm \lambda} - d _{i-1} ^{\bm \lambda}$. By Theorem \ref{equigen} and Corollary \ref{dim_change}, it suffices to prove $\dim B \mathbb V ^{{\bm \lambda}} \ge \left| \mu \right| + 2 \sum _{i < j} d _i ^{\nabla} d _j ^{\nabla}$. The subspace $\mathbb V ^{{\bm \lambda}} _{1}$ is $B$-stable and has dimension $\left| \mu \right|$. Since $B \mathbb V ^{{\bm \lambda}} = B \mathbb V ^{{\bm \lambda}} _{01}$, we have only to prove $\dim B \mathbb V ^{{\bm \lambda}} _{0} \ge 2 \sum _{i < j} d _i ^{\nabla} d _j ^{\nabla}$. Here $\mathbb V ^{{\bm \lambda}} _{0}$ is $N _0$-stable and $\dim \mathbb V ^{{\bm \lambda}} _{0} = \sum _{i < j} d _i  ^{\nabla} d _j ^{\nabla}$. Thus, it suffices to prove that $\dim G _2 x \ge \sum _{i < j} d _i ^{\nabla} d _j ^{\nabla}$ for a generic element $x \in \mathbb V ^{{\bm \lambda}} _{0}$. Since the dimension of the $G _2$-stabilizer is an upper semi-continuous function along $\mathbb V ^{{\bm \lambda}} _{0}$, it is enough to show $\dim G _2 x \ge \sum _{i < j} d _i ^{\nabla} d _j ^{\nabla}$ for some $x \in \mathbb V ^{{\bm \lambda}} _{0}$. For $x \in \mathbb V _0$, we have
$$\dim \mathfrak g _2 x = \dim \mathfrak g _{-2} x.$$
Theorem \ref{dim_formula} implies that
$$\dim \mathcal O _{( \lambda, \mathbf 0 )} = \dim \mathfrak g x = 2 \dim \mathfrak g _0 x = 4 \sum _{i < j} d _i ^{\nabla} d _j ^{\nabla}.$$
In particular, we have $\dim \mathfrak g _2 x = \sum _{i < j} d _i ^{\nabla} d _j ^{\nabla}$ as desired.
\end{proof}

\begin{lemma}\label{orbtoA}
Let ${\bm \lambda} = ( \lambda, a )$ be a marked partition. Then, we have an equality
$$\overline{\mathfrak X _{{\bm \lambda}}} = \mathbb V _1 ^{{\bm \lambda}} \oplus \overline{\mathfrak X}$$
for some orbital variety $\mathfrak X$ of $\mathcal O _{(\lambda, {\mathbf 0})}$.
\end{lemma}

\begin{proof}
The number $d ^{\nabla} _i$ from the proof of Lemma \ref{irrcomp} is some parts of ${}^t( \mu ^{\bm \lambda})$ if $i \le \mu _1 ^{\bm \lambda}$ and count some parts of ${}^t( \nu ^{\bm \lambda})$ if $i > \mu _1 ^{\bm \lambda}$. We define $\{ d _i ^{\prime} \}$ to be the parts of the partition ${}^t \lambda$. Since $\lambda _j = \mu ^{\bm \lambda} _j + \nu ^{\bm \lambda} _j$ for all $j$, we deduce that the two sequences $\{d _i ^{\nabla} \} _i$ and $\{ d _i ^{\prime} \} _i$ coincides up to changing their ordering.

Here $\{ d _i ^{\prime} \} _i$ is a decreasing sequence. It follows that the dense open part of $\mathbb V _0 ^{\bm \lambda}$ has the same Jordan normal form as that of $\mathbb V _0 ^{(\lambda, \mathbf 0)}$. Here, $\mathbb V _1 ^{\bm \lambda}$ is $B$-stable. Therefore, Lemma \ref{irrcomp} and Corollary \ref{dim_change} implies that
$$\dim \overline{\mathfrak X} = \dim \overline{\mathfrak X _{{\bm \lambda}}} - d _{\mu _1 ^{\bm \lambda}} = \frac{1}{2} \dim \mathcal O _{\bm \lambda} - | \mu ^{\bm \lambda} | = \frac{1}{2} \dim \mathcal O _{( \lambda, \mathbf 0)}.$$
By Theorem \ref{equigen}, the closure $\overline{B \mathbb V_0 ^{\bm \lambda}}$ must be the closure of an orbital variety of $\mathcal O_{(\lambda, \mathbf 0)}$ as desired.
\end{proof}

\begin{definition}[Special vectors]
Let $\bm \lambda$ be a marked partition of $n$ and let $(\mu ^{\bm \lambda}, \nu ^{\bm \lambda})$ be its associated bi-partition (c.f. Theorem \ref{mp<->bp} and the line below it). We define
$$\mathsf{D} ^0 _i ( {\bm \lambda} ) := \prod _{d _i < k < l \le d _{i + 1}} ( \epsilon _k ^2 - \epsilon _l ^2 ), \text{ and } \mathsf{D} ^+ _i ( {\bm \lambda} ) := \mathsf{D} ^0 _i ( {\bm \lambda} ) \prod _{d _i < k \le d _{i + 1}} \epsilon _k.$$
Using this, we define
$$\mathsf{D} ( \mu ^{\bm \lambda}, \nu ^{\bm \lambda} ) := \prod _{i = 0} ^{\mu _1 - 1} \mathsf{D} ^0 _i ( {\bm \lambda} ) \times \prod _{i = \mu _1} ^{\mu _1 + \nu _1 - 1} \mathsf{D} ^+ _i ( {\bm \lambda} ).$$
\end{definition}

For an arbitrary bi-partition $(\mu, \nu)$ of $n$, we define the Macdonald representation attached to $( \mu, \nu )$ as
$$L ( \mu, \nu ) := \mathbb C [ W ] \mathsf{D} ( \mu, \nu ) \subset \mathbb C [ \mathfrak t ].$$
We remark that $L (\mu, \nu)$ is well-defined due to Theorem \ref{mp<->bp}. Let $\mathbb C [ \mathfrak t ] _m$ denote the degree $m$-part of the polynomial ring $\mathbb C [\mathfrak t ]$. For a subset $I \subset \mathbb C [ \mathfrak t ]$, we put $I _m := I \cap \mathbb C [ \mathfrak t ] _{m}$.

\begin{theorem}[Macdonald c.f. Lusztig-Spaltenstein \cite{LS}]
For each bi-partition $( \mu, \nu )$ of $n$, the Macdonald representation $L ( \mu, \nu )$ is an irreducible representation of $W$. Moreover, we have
$$\Hom _W ( L ( \mu, \nu ), \mathbb C [ \mathfrak t ] _{m} ) = \begin{cases} 0 & (m < \deg \mathsf{D} (\mu, \nu)) \\ 1 & (m = \deg \mathsf{D} (\mu, \nu)) \end{cases}.$$
\end{theorem}

We define
$$W _l := \left< s _i s _{i + 1} \cdots s _{n - 1} s _n s _{n - 1} \cdots s _i ; 1 \le i \le n \right> \subset W.$$
We have $W _l \cong ( \mathbb Z / 2 \mathbb Z ) ^n$. Moreover, we have a short exact sequence of groups
$$\{ 1 \} \longrightarrow W _l \longrightarrow W \longrightarrow \mathfrak S _n \longrightarrow \{ 1 \}.$$

\begin{corollary}\label{descEX}
Let $\lambda$ be a partition of $n$. For each integer $m$, we have
\begin{enumerate}
\item $\Hom _{W} ( L ( \lambda, \emptyset ), \mathbb C [ \mathfrak t ] _{m} ) = \Hom _{\mathfrak S _n} ( L ( \lambda, \emptyset ), \mathbb C [ \epsilon _1 ^2, \ldots, \epsilon _n ^2 ] _{m} )$;
\item $\Hom _{W} ( L ( \emptyset, \lambda ), \mathbb C [ \mathfrak t ] _{m} ) = \Hom _{\mathfrak S _n} ( L ( \emptyset, \lambda ), \left( \epsilon _1 \epsilon _2 \cdots \epsilon _n \mathbb C [ \epsilon _1 ^2, \ldots, \epsilon _n ^2 ] \right) _{m} )$.
\end{enumerate}
\end{corollary}

\begin{proof}
An irreducible $W$-module for which $W _l$ acts trivially yields a $\mathfrak S _n$-module. The invariant ring $\mathfrak C [ \mathfrak t ] ^{W _l}$ is generated by $\epsilon _1 ^2, \ldots, \epsilon _n ^2$. This implies the first assertion. The whole ring $\mathfrak C [ \mathfrak t ]$ is a free $\mathfrak C [ \mathfrak t ] ^{W _l}$-module of rank $2 ^n$. The sign representation of $W _l$ also defines a rank one $\mathfrak C [ \mathfrak t ] ^{W _l}$-submodule $\mathfrak C [ \mathfrak t ] ^{W _l} ( \epsilon _1 \epsilon _2 \cdots \epsilon _n ) \subset \mathfrak C [ \mathfrak t ]$ via taking the isotypical component. Hence, we deduce the second assertion.
\end{proof}

\begin{theorem}\label{MacG}
Let ${\bm \lambda}$ and $( \mu ^{{\bm \lambda}}, \nu ^{{\bm \lambda}})$ be a marked partition of $n$ and its associated bi-partition. Let $\mathfrak X _{\bm \lambda}$ be the orbital variety of $\mathcal O _{\bm \lambda}$ defined in Lemma \ref{irrcomp} $($c.f. Definition \ref{mp<->bp}$)$. We have
$$\mathsf{J} ( \mathfrak X _{{\bm \lambda}} ) \in \mathbb Q \mathsf{D} ( \mu ^{{\bm \lambda}}, \nu ^{{\bm \lambda}} ).$$
In particular, the $\mathbb Q$-span of $\mathrm{Jos} ( \mathcal O _{{\bm \lambda}} )$ is equal to $L ( \mu ^{{\bm \lambda}}, \nu ^{{\bm \lambda}})$ as $W$-submodules of $\mathbb Q [\mathfrak t]$.
\end{theorem}

\begin{proof}[Proof of Theorem \ref{MacG}]
We define
$$q _0 ( {\bm \lambda} ) := \prod _{i > \left| \mu \right|} \epsilon _i \times \prod _{i \ge 0} \prod_{d _i < j < k \le d _{i + 1}} ( \epsilon _j - \epsilon _k ).$$
The map $\texttt{pr} : \mathbb V ^+ \longrightarrow \mathbb V ^+ / ( \mathbb V _1 \oplus \mathbb V _2 )$ is a $B$-equivariant fibration. Hence, it induces the associated map $\mathfrak X _{{\bm \lambda}} \longrightarrow \texttt{pr} ( \mathfrak X _{{\bm \lambda}} )$, which is generically a flat fibration. By Lemma \ref{basic identity}, we have $\overline{\texttt{pr} ( \mathfrak X _{{\bm \lambda}} )} = \mathbb V ^{\bm \lambda} \cap \mathbb V _0$. Moreover, each $\overline{\mathfrak X _{\bm \lambda}}$ is written as a product of $\bigoplus _{1 \le i \le \left| \mu \right|} V _1 [ \epsilon _i ]$ and $\overline{\mathfrak X _{\bm \lambda}} \cap V _2$. It follows that
\begin{eqnarray}
q _0 ( {\bm \lambda} ) \text{ divides } \mathsf{J} ( \mathfrak X _{{\bm \lambda}} ).\label{sdivide}
\end{eqnarray}

By a dimension counting, we deduce that
\begin{eqnarray}
d _Y := \dim \mathcal E _Y = \frac{1}{2} \mathrm{codim} _{\mathfrak N} G Y = \deg \mathsf{J} ( \mathfrak X _{{\bm \lambda}} ) = \deg \mathsf{D} ( \mu, \nu )\label{degcomp}
\end{eqnarray}
for each $Y \in \mathcal O _{\bm \lambda}$.

Applying the argument of \cite{CG} 6.5.3 and 7.4.1, we know that
$$H _{2 d _Y} ( \mathcal E _Y, \mathbb C ) \hookrightarrow H _{\bullet} ( \mathcal B, \mathbb C ) \hookrightarrow \mathbb C [ \mathfrak t ]$$
is some Macdonald representation. (The second inclusion is realized by the harmonic polynomials determined by $T$-equivariant fundamental classes.) By \cite{K1} Theorem 8.3 and \cite{CG} \S 8.9, this establishes a one-to-one correspondence between the set of Macdonald representations and the set of $G$-orbits of $\mathfrak N$.

By Lemma \ref{orbtoA}, we have $\overline{\mathfrak X _{\bm \lambda}} \cap V _2 = \overline{\mathfrak X _{\bm \lambda} ^{\sim}}$ for some $\mathfrak X _{\bm \lambda} ^{\sim} \in \mathrm{Comp} ( \mathcal O_{(\lambda, \mathbf 0 )})$. Hence, we have an equality
\begin{eqnarray}
( \prod _{i \le \left| \mu \right|} \epsilon _i ) \mathsf{J} ( \mathfrak X _{{\bm \lambda}} ) = \mathsf{J} ( \mathfrak X _{\bm \lambda} ^{\sim} ).\label{linkingL}
\end{eqnarray}
For $f \in \mathbb C [ \epsilon _1 ^2, \ldots, \epsilon _n ^2 ]$, it is standard to see
\begin{eqnarray}
( \epsilon _i - \epsilon _j )\text{ (resp. $\epsilon _i$) divides $f$ if and only if $( \epsilon _i + \epsilon _j )$ (resp. $\epsilon _i$) divides $f$}.\label{GAL}
\end{eqnarray}
By Corollary \ref{descEX}, if $\mathsf{J} ( \mathfrak X _{{\bm \lambda}} ) \in L ( \gamma, \emptyset )$ or $\mathsf{J} ( \mathfrak X _{{\bm \lambda}} ) \in L ( \emptyset, \gamma )$ for a partition $\gamma$, then we have
\begin{eqnarray}
\mathsf{J} ( \mathfrak X _{{\bm \lambda}} ) \in \mathbb C D ( \mu ^{\bm \lambda}, \nu ^{\bm \lambda} )\label{weakpa}
\end{eqnarray}
by a degree counting. 

We prove the condition (\ref{weakpa}) by induction on the usual partition $\lambda$ of $n$ constituting ${\bm \lambda} = ( \lambda, a )$.

\begin{claim}
If $\lambda = (1,1,\ldots,1)$, then (\ref{weakpa}) holds.
\end{claim}
\begin{proof}
We have either $a = \{ 0 \}$ or $\{ 1, 0, \cdots \}$. If $a = \{ 0 \}$, then the weight configuration of $\mathbb V ^+$ is the same as the positive roots of $\mathop{SO} ( 2n + 1 )$. It follows that $\mathsf{J} ( \mathfrak X _{{\bm \lambda}} ) = D ( \emptyset, \lambda )$, which implies $\mathsf{J} ( \mathfrak X _{{\bm \lambda}} ) \in L ( \emptyset, \lambda )$. If $a = \{ 1, 0, \cdots \}$, then we have $\mathsf{J} ( \mathfrak X _{{\bm \lambda}} ) = D ( \lambda, \emptyset ) \in L ( \lambda, \emptyset )$. This proves (\ref{weakpa}) in all cases.
\end{proof}

\begin{claim}\label{indL}
Assume that (\ref{weakpa}) holds for $\bm \lambda = ( \lambda, {\bm 0} )$. Then, we have (\ref{weakpa}) for all marked partition of shape $(\lambda, *)$.
\end{claim}

\begin{proof}
We have $L ( \mu ^{( \lambda, \mathbf 0 )}, \nu ^{( \lambda, \mathbf 0 )} ) \subset \mathbb C [ \mathfrak t ] ^{W _l} ( \epsilon _1 \cdots \epsilon _n )$. Hence, for every marked partition of shape $(\lambda, a ^{\prime})$, we have $\mathsf{J} ( \mathfrak X _{(\lambda, a ^{\prime})} ^{\sim} ) \in \mathbb C [ \mathfrak t ] ^{W _l} ( \epsilon _1 \cdots \epsilon _n )$. Hence, we deduce (\ref{weakpa}) for $(\lambda, a ^{\prime})$ from (\ref{sdivide}), (\ref{linkingL}), and (\ref{GAL}).
\end{proof}

We return to the proof of Theorem \ref{MacG}

Assume that we know (\ref{weakpa}) for all $\bm \lambda = ( \lambda, a )$ when $\deg D ( \mu ^{( \lambda, \mathbf 0 )}, \nu ^{( \lambda, \mathbf 0 )} ) > M$. We prove the case $\deg D ( \mu ^{( \lambda, \mathbf 0 )}, \nu ^{( \lambda, \mathbf 0 )} ) = M$. We a priori knows that $\{ L ( \mu, \nu ) \}_{(\mu, \nu)}$ gives a complete collection of $W ^{\vee}$. Therefore, the induction hypothesis implies that each $D ( \mu ^{\bm \lambda}, \nu ^{\bm \lambda} ) \in \mathbb C [ \mathfrak t ] _M$ which we have not yet proved (\ref{weakpa}) admits a factorization by $( \epsilon _1 \ldots \epsilon _n )$ and a $W _l$-invariant polynomial. Hence, we deduce (\ref{weakpa}) for $\bm \lambda = ( \lambda, \mathbf 0 )$ such that $\deg D ( \mu ^{( \lambda, \mathbf 0 )}, \nu ^{( \lambda, \mathbf 0 )} ) = M$ from (\ref{sdivide}) and (\ref{GAL}). Now Claim \ref{indL} proceeds the induction and we obtain (\ref{weakpa}) for every marked partition.

Since an equivariant Hilbert polynomial has rational coefficients, we deduce 
$$\mathsf{J} ( \mathfrak X _{{\bm \lambda}} ) \in \mathbb Q D ( \mu ^{\bm \lambda}, \nu ^{\bm \lambda} )$$
as desired.
\end{proof}

{\bf Acknowledgement:} The author would like to thank Michel Brion for pointing out a shortcut presented in \S \ref{nnt}. The author also like to thank Corrado De Concini, Andrea Maffei, Hisayosi Matumoto, Toshiaki Shoji, Tonny A. Springer, Toshiyuki Tanisaki for their interests and remarks. George Lusztig carried an attention to the paper \cite{X} to the author. The author want to express his gratitude to all of them. Part of this paper is written during his stay at MPIfM Bonn. The author thanks their hospitality and comfortable atmosphere. Finally, S.K. likes to express his thanks to Midori Kato for her constant support.

\end{document}